\documentclass[12pt]{amsart}
\usepackage{amsthm,amsfonts,amssymb,amsmath,oldgerm}
\numberwithin{equation}{section}
\usepackage{algpseudocode}
\usepackage{algorithm2e}
\usepackage{fullpage}
\usepackage{amsmath}
\usepackage{setspace}
\usepackage{stmaryrd}
\usepackage{mathrsfs}
\usepackage{fancyhdr}
\usepackage{graphicx}
\usepackage{enumerate}
\usepackage{bm} 
\usepackage{bbm}
\usepackage{dsfont}
\usepackage{multirow}
\usepackage{psfrag}
\usepackage[font=scriptsize]{caption}
\usepackage[font=scriptsize]{subcaption}
\usepackage{listings}
\usepackage{tikz}
\usepackage{empheq}
\usepackage{cases}
\usetikzlibrary{matrix} 
\usepackage[framemethod=TikZ]{mdframed}
\mdfdefinestyle{MyFrame}{backgroundcolor=gray!20!white}
\usepackage[makeroom]{cancel}
\usepackage[top=30mm,bottom=30mm,left=22mm,right=22mm,a4paper]{geometry}
\usetikzlibrary{fit}
\usetikzlibrary{arrows}
\usepackage{hyperref,bookmark}
\numberwithin{equation}{section}
\hypersetup{linkbordercolor=red}
\usepackage{color}
\let\eps\varepsilon
\def\be{\begin{equation}}
\def\ee{\end{equation}}
\let\ds\displaystyle
\let\pa\partial

\def\RR{\mathbb{R}}
\def\TT{\mathbb{T}}
\def\NN{\mathbb{N}}

\def\cC{\mathcal C}

\def\cH{\mathcal H}
\def\cI{\mathcal I}
\def\cM{\mathcal M}

\def\cQ{\mathcal Q}

\def\bfu{\mathbf{u}}
\def\bfv{\mathbf{v}}
\def\bfx{\mathbf{x}}
\def\bfq{\mathbf{q}}
\def\bfp{\mathbb{P}}
\def\bfE{\mathbf{E}}
\def\bfB{\mathbf{B}}

\def\bfU{\mathbf{U}}

\newtheorem{remark}{Remark}[section]

\newtheorem{prop}{Proposition}[section]
\newtheorem{theorem}{Theorem}[section]

\newcommand{\cla}[1]{\color{black}#1 \color{black}}

\def\pa{\partial}

\newcommand\dD{\mathrm{d}}
\newcommand\Div{\mathrm{div}}
\title[Multi-species Fokker-Planck equation]{Fokker-Planck multi-species equations in the adiabatic asymptotics}

\author[F. Filbet and C. Negulescu]{Francis Filbet and Claudia Negulescu}

\address{Universit\'e de Toulouse \& CNRS, UPS, Institut de Math\'ematiques de Toulouse UMR 5219, F-31062 Toulouse, France.}
\email{francis.filbet@math.univ-toulouse.fr, claudia.negulescu@math.univ-toulouse.fr}
\date{\today}

\begin{document}
\maketitle

\begin{abstract}
{The main concern of the present paper is the study of the multi-scale dynamics of thermonuclear fusion plasmas via a multi-species Fokker-Planck kinetic model. 
One of the goals is the generalization of the standard Fokker-Planck
collision operator to a multi-species one, conserving mass, total
momentum and energy, as well as satisfying Boltzmann's
$H$-theorem. Secondly, we perform on one hand a mathematical
asymptotic limit, letting the electron/ion mass ratio converging towards zero,  to obtain a thermodynamic
equilibrium state for the electrons (adiabatic regime), whereas the ions are kept
kinetic.   On the other hand, we develop a first numerical scheme,
based on a Hermite spectral method, and perform numerical simulations
to  investigate in more details this asymptotic limit.}
\end{abstract}

\bigskip


\maketitle

\tableofcontents
%%%%%%%%%%%%%%%%%%%%%%%%%%%%%%%%%%%%%%%%%%%%%%
\section{Introduction}
\label{SEC1}
\setcounter{equation}{0}
\setcounter{figure}{0}
\setcounter{table}{0}

%%%%%%%%%%%%%%%%%%%%%%%%%%%%%%%%%%%%%%%%%%%%%%%
%%%%%%%%%%%%%%%%%%%%%%%%%%%%%%%%%%%%%%%%%%%%%%%
Starting with the first projects born in Russia in the early 1950's, continuous efforts were made to produce clean and reliable energy in tokamak fusion reactors able to confine a very hot plasma gas via strong electromagnetic fields. The mathematical modelling is a useful tool in this process.
Kinetic models, based on a mesoscopic description of the various
particles constituting a plasma, and coupled to Maxwell's equations
for the computation of the electromagnetic fields, are very precise
approaches for the study of such thermonuclear fusion
plasmas. However, treating each species in a kinetic framework is
computationally very demanding, such that approximate models
have been introduced. Especially when one is
interested in the investigation of phenomena occurring on the (slow)
ion scales, electrons are approximated via macroscopic models
(adiabatic models). The justification is that the time- and
length-scales associated with the electrons are very small as compared
to the ones of the ions (due to the small mass ratio $m_e/m_i$), such
that electrons are considered to be in a
thermodynamic equilibrium. Such hybrid strategies, treating ions in a fully kinetic manner and
electrons via fluid approaches, are often used in today's simulations
\cite{cartwright,PHYS2,garbet,kwok}, leading to significant savings in
computational time and memory. But, describing particles via a fluid
model requires that the electron distribution function remains close to a
thermodynamic equilibrium, meaning being close to a Maxwellian distribution
in the velocity variable. Coulomb collisions in a thermonuclear
plasma are however not sufficiently effective to thermalize the
electrons. Thus, the validity of the adiabatic electron model
(electron Boltzmann relation) is rather controversial. Indeed, this
model seems to break down in various situations, as for example in the
edge plasma region, or it does not take into account for
important instabilities, such as the Trapped Electron Modes (TEM), which
are considered as essential in the turbulent dynamics \cite{dom}.\\

In this paper, we shall especially  focus on the asymptotic towards the electron Boltzmann
regime, starting from a kinetic picture where
  collisions and collective effects (electrostatic forces) are well
  balanced. To investigate this dynamic,  we are firstly introducing a multi-species
Fokker-Planck equations, with particular emphasize on the
inter-species collision operators and their properties.
In the literature one can find various simplified models for inter-species
collisions. For instance BGK models for gas mixtures are given as  a
sum of relaxation operators. One example is the model of Klingenberg,
Pirner and Puppo \cite{KPP} or Bobylev, Bisi, Groppi,
Spiga and Potapenko \cite{Bob}. It contains the often used
models of Gross and Krook \cite{G73,GK} and Hamel \cite{H} as special
cases. Other type of models contain only one collision term on the
right-hand side as the one proposed by Andries, Aoki and Perthame in
\cite{AAP}. In this paper we focus on Fokker-Planck type operators which are
more consistent for the description of collisional plasmas \cite{Degond_rev, taitano,cruz,duclous}.  The model
is derived by introducing mixed temperatures and momenta, under the
constraint that the number of particles of each species, the total
momentum and the total energy are conserved. Moreover, we prove that the
model satisfies an H-Theorem, permitting to characterize the form of equilibrium.

Having introduced these Fokker-Planck collision operators, a physical scaling is performed permitting to characterize the regime of
interest in our plasma studies, namely focusing on the ion dynamics. This allows to identify a small parameter $\eps \ll 1$ which shall permit to obtain the desired electron Boltzmann relation, when performing a formal asymptotic limit $\eps \rightarrow 0$. 
{Our main goal is to design a numerical method
  able to give precise results for all  values of $\eps \in[0,1]$,
  especially able to follow the asymptotic limit $\eps \rightarrow 0$
  without  extensive numerical efforts. The idea is to have a scheme
  which can treat electrons and ions simultaneously without having to
  adapt the mesh to the different species, but rather to the physical
  phenomenon one wants to investigate. In this aim we shall present in
  this paper the first step towards such a performant scheme, based on
  a Hermite spectral approach to cope with the velocity variable \cite{Filbet2020, Filbet2022, SH}. Hermite spectral methods offer indeed an ideal way to perform
  large-scale simulations, including at the same time microscopic
  kinetic effects. The choice of a suitable scaling of the Hermite
  basis functions, adapted to the investigated asymptotic, is
  fundamental, rendering the Hermite approach intrinsically multiscale
  and providing thus a natural bridge between the microscopic and the
  macroscopic worlds. Indeed, in the limit $\eps\rightarrow 0$,  the
  distribution function can be represented by only one Hermite
  function, reducing drastically the number of discretization parameters in the
  velocity space.  
The use of Hermite functions for the resolution of kinetic equations
was proposed for the first time by Grad in \cite{Grad}. }

{At the end of this work numerical simulations are presented in the aim to show  the advantage of such a Hermite spectral approach, in particular to understand
how the electron distribution function converges after a transient
regime towards its thermodynamic limit in the context of plasma
simulations.} %The here presented method is not {\it
              %Asymptotic-Preserving} as the time-step $\Delta t$ has
              %still to be adapted to the electron dynamics, thus
              %$\Delta t \sim \eps$ when the solution is far from its
              %equilibrium.
The specificity of this method will be underlined, namely the fact that it permits considerable  improvements in simulation time for small $\eps$-values, as in such regimes very few Hermite modes have to be taken into account.\\ %A second step has still to be performed to eliminate the $\eps$-dependence of the time-step $\Delta t$ and this will be the aim of a second work in progress for the moment \cite{CF_bis}, where a fully AP-scheme shall be presented and more realistic physical simulations performed, possible due to the $\eps$-independent meshes.\\

The outline of this paper is the following. In Section \ref{SEC2} we
present the fully kinetic ion-electron model and its physical scaling
leading to a multi-scale multi-species coupled Fokker-Planck
model. Section \ref{SEC3} deals with the formal derivation of the
hybrid limit-model, when letting a small parameter $\eps \ll 1$
tend towards zero, parameter standing somehow for the small
electron-to-ion mass ratio. The well-posedness of the limit-model is
also investigated. And finally in Sections \ref{sec:4} and \ref{sec:5} we shall present a first numerical scheme, based on a Hermite spectral approach, and shall conclude with the study of some numerical simulations. 
%%%%%%%%%%%%%%%%%%%%%

\section{The mixed kinetic model and its scaling}
\label{SEC2}
\setcounter{equation}{0}
\setcounter{figure}{0}
\setcounter{table}{0}

%%%%%%%%%%%%%%%%%%%%% 

%%%%%%%%%%%%%%%%%%%%%
\subsection{The mixed kinetic model}
%%%%%%%%%%%%%%%%%%%%%

The starting point of this work is the  following model composed of two coupled kinetic equations for the ions respectively  electrons of a fusion plasma, {\it i.e.}
\begin{equation}
\left\{
 \begin{array}{ll}
      \displaystyle \partial_t f_i \,+\, \bfv\cdot \nabla_\bfx f_i \,+\, \frac{e}{m_i}\, \bfE\cdot \nabla_\bfv f_i \,=\, \cQ_{ii}(f_i,f_i) \,+\, \cQ_{ie}(f_i,f_e)\,, \\[1.1em]
      
      \displaystyle \partial_t f_e \,+\, \bfv\cdot  \nabla_\bfx f_e \,-\, \frac{e}{m_e}\, \bfE \cdot \nabla_\bfv f_e \,=\, \cQ_{ee}(f_e,f_e) \,+\, \cQ_{ei}(f_e,f_i)\,,
    \end{array}
  \right.
\label{SystemKineticDimensional}  
\end{equation}
associated to Poisson's equation for the description of the electrostatic potential
\begin{equation}
- \Delta \phi \,=\, \frac{e}{\varepsilon_0}\,(n_i \,-\, n_e), \qquad \bfE \,=\, - \nabla_\bfx \phi\,,
\label{PoissonDim}    
\end{equation}
with $e$ the elementary charge, $m_{\alpha}$ the  mass of species $\alpha\in\{e,\,i\} $ and $\varepsilon_0$ the vacuum permitivity. The magnetic field $\bfB$ is not considered here, as we shall focus in the following rather on the dynamics  parallel to $\bfB$ and did not want to encumber the paper. For a more general framework see \cite{cla}.
The functions $f_{\alpha}(t,\bfx,\bfv)$ represent the particle distribution functions in the phase-space $(\bfx,\bfv)\in  {\TT^d} \times \mathbb{R}^d$ (${\TT}^d$ being the $d$-dimensional  torus) whereas the electron and ion macroscopic quantities are given for $\alpha\in\{e,\,i\}$ by
\begin{equation*}
  %\label{MACRO_q}
  \left\{
\begin{array}{l}	
    \displaystyle n_{\alpha}(t,\bfx)  \,:=\,   \int_{\mathbb{R}^d} f_{\alpha}(t,\bfx,\bfv)\, \dD\bfv\,,  \\[1.1em]
    \displaystyle n_{\alpha} \bfu_{\alpha}(t,\bfx)  \,:= \,  \int_{\mathbb{R}^d} \bfv\, f_{\alpha}(t,\bfx,\bfv)\, \dD \bfv\,, \\[1.1em]
    \displaystyle  d\, k_B \;n_{\alpha}\; T_{\alpha}(t,\bfx)  \, := \,   m_{\alpha}\, \int_{\mathbb{R}^d} |\bfv-\bfu_{\alpha}|^2\, f_{\alpha}(t,\bfx,\bfv)\, \dD\bfv\,, \\[1.1em]
    \displaystyle  w_{\alpha}(t,\bfx)  \, := \,  {m_{\alpha} \over 2}\, \int_{\mathbb{R}^d} |\bfv|^2\, f_{\alpha}(t,\bfx,\bfv)\, \dD\bfv = \frac{d}{2}\, k_B n_{\alpha} T_{\alpha} + {m_{\alpha} \over 2} n_{\alpha} |\bfu_{\alpha}|^2\,,
\end{array}\right.
\end{equation*}
with $k_B$ the Boltzmann constant. The collision operators describing the  inter- and intra-species interactions are chosen of Fokker-Planck type, {\it i.e.} given for $\alpha$, $\beta \in \{e,i\}$ by
\begin{equation*}
\displaystyle
\cQ_{\alpha\beta}(f_\alpha,f_\beta) \,:=\, \nu_{\alpha \beta}\,\Div_\bfv\left( (\bfv - \bfu_{\alpha\beta}) f_\alpha + \frac{k_B T_{\alpha\beta}}{m_\alpha} \nabla_\bfv f_\alpha \right)\,,
%\label{definitionOperateurCollision}
\end{equation*}
where $\nu_{\alpha \beta}>0$ are the collisional frequencies corresponding to the couple $(\alpha,\beta)$ of particles.

{For intra-species collisions, we have $\bfu_{\alpha \alpha} := \bfu_\alpha$, $T_{\alpha \alpha }:= T_{\alpha}$ and the} collision operator $\cQ_{\alpha\alpha}(f_\alpha,f_\alpha)$ is chosen in such a way to get conservation of mass, momentum and kinetic energy
$$
m_\alpha\,\int_{\mathbb{R}^d} \cQ_{\alpha\alpha}(f_\alpha,f_\alpha)\,\left(
\begin{array}{l}
 \, \,1\\
  \,\,\bfv\\
\ds\frac{|\bfv|^2}{2}
\end{array}\right)\, \dD\bfv \,=\, 0 \,,\qquad \forall \alpha\in\{e,\,i\}\,,
$$
as well as the entropy decay
\begin{equation*}
         \int_{\mathbb{R}^d} \cQ_{\alpha\alpha}(f_\alpha,f_\alpha) \,\ln(f_\alpha) \,\dD\bfv\, \,\leq\, 0\,,\qquad \forall \alpha\in\{e,\,i\}\,,
  %      \label{EntropyDecayUnmix}
\end{equation*}
leading to the thermal equilibrium
\begin{equation*}
  \int_{\mathbb{R}^d} \cQ_{\alpha\alpha}(f_{\alpha},f_{\alpha}) \ln(f_{\alpha})\, \dD\bfv \,=\, 0 \quad \Longleftrightarrow \quad f_{\alpha}\,=\,\cM_{n_{\alpha},\bfu_{\alpha},T_{\alpha}}\,,
\end{equation*}
where $\cM_{n_{\alpha},\bfu_{\alpha},T_{\alpha}}$ is the local Maxwellian defined by
\be
\label{localM}
\cM_{n_{\alpha},\bfu_{\alpha},T_{\alpha}}(\bfv) := n_\alpha\, \left( \frac{m_\alpha}{2 \pi\, k_B\, T_{\alpha}} \right)^{d/2}\, \exp\left(-m_\alpha \frac{|\bfv - \bfu_{\alpha}|^2}{2 k_B T_{\alpha}}\right).
\ee

However for the inter-species collisions the situation is more complex.  {The choice of the inter-species mixed velocities $\bfu_{\alpha \beta}$ and temperatures $T_{\alpha \beta}$ is done such that to enforce the appropriate conservation laws and to ensure the H-theorem. For this we shall first of all require that
\begin{equation}
\bfu_{\alpha \beta}=\bfu_{\beta \alpha}\,, \qquad T_{\alpha \beta}=T_{\beta \alpha}\,, \qquad \nu_{ei} m_e n_e = \nu_{ie} m_i n_i\,.
\label{relationCollisionFrequency}
\end{equation}
These three requirements are fundamental and also physical. The justification of the last assumption comes from the Coulomb collisional frequency  \cite{hazel}, given by
$$
\nu_{\alpha \beta}= C\, e_\alpha^2\, e_\beta^2\, n_\beta\, {m_\beta \over m_\alpha + m_\beta} \, {1 \over (v_{th,\alpha}^2+v_{th,\beta}^2)^{3/2}}\,, \quad C>0\,.
$$
With these assumptions and the fact that we would like to ensure the  total  momentum conservation
       \begin{equation*}
        m_e \int_{\mathbb{R}^d} \cQ_{ei}(f_e,f_i) \,\bfv\, \dD\bfv \,+\, m_i \int_{\mathbb{R}^d} \cQ_{ie}(f_i,f_e) \,\bfv \,\dD \bfv \,=\, 0\,,
 %       \label{MomentumMixProperty}
    \end{equation*}
the total kinetic energy conservation
       \begin{equation*}
        m_e \int_{\mathbb{R}^d} \cQ_{ei}(f_e,f_i) \,{{|\bfv|^2}\over{2}} \,\dD\bfv \,+\, m_i \int_{\mathbb{R}^d} \cQ_{ie}(f_i,f_e) \,{{|\bfv|^2}\over{2}} \,\dD\bfv \,=\, 0\,,
 %       \label{kineticEnergyMixProperty}
    \end{equation*}
as well as the global entropy decay
       \begin{equation*}
         \int_{\mathbb{R}^d} \cQ_{ei}(f_e,f_i) \,\ln(f_e) \,\dD\bfv  \,+\, \int_{\mathbb{R}^d} \cQ_{ie}(f_i,f_e) \ln(f_i) \,\dD\bfv \,\leq\, 0\,,
    \end{equation*}
a unique choice of mixed velocities is possible, given by
\begin{equation}
 \bfu_{ei} \,=\, \bfu_{ie} \,:=\, \frac{\bfu_e + \bfu_i}{2}\,,
\label{mixed1}
\ee
as well as a unique choice of mixed temperatures, namely
\be
 T_{ei} \,=\, T_{ie} \,:=\, \frac{m_i\, T_e \,+\, m_e\, T_i}{m_e + m_i} \,+\, \frac{m_i \,m_e}{m_i \,+\, m_e} \frac{|\bfu_e - \bfu_i|^2}{2 d \,k_B}\,.
\label{mixed2}
\ee
To simplify the computations, let us remark that the collision operators} can be rewritten in a simpler form as follows
\begin{equation*}
\cQ_{\alpha \beta}(f_\alpha,f_\beta) \,=\, \nu_{\alpha \beta} \,\,\Div_\bfv\left( \frac{k_B T_{\alpha \beta}}{m_\alpha}\,\cM_{\alpha \beta}\, \nabla_\bfv \left( \frac{f_\alpha}{\cM_{\alpha \beta}} \right)  \right)\,,
%\label{definitionOperateurMaxwel}
\end{equation*}
with $\cM_{\alpha \beta}$ the local Maxwellian given by

\be \label{oufff}
\cM_{\alpha \beta}(t,\bfx,\bfv) := n_\alpha(t,\bfx)\, \left( \frac{m_\alpha}{2 \pi\, k_B\, T_{\alpha \beta}(t,\bfx)} \right)^{d/2}\, \exp\left(-m_\alpha \frac{|\bfv - \bfu_{\alpha \beta}(t,\bfx)|^2}{2 k_B T_{\alpha \beta}(t,\bfx)}\right).
%\label{definitionDiffMaxwellian}
\ee

Unlike Boltzmann's operators for neutral gases, the Fokker-Planck
operator expresses the cumulative effects of many grazing collisions
(rather than short-range collisions), and this is due to the long-range effect of the Coulomb interactions. Thus Fokker-Planck operators describe mainly a diffusion in the velocity space.

%%%%%%%%%%%%%%%%%%%%%
\subsection{Characteristic scales and regime of interest} \label{SEC_Ca}
%%%%%%%%%%%%%%%%%%%%%

Let us now identify some small parameters, characterizing the adiabatic regime of  plasma dynamics. This shall be done by firstly introducing the orders of magnitude of the quantities involved in the description of the phenomenon we want to analyse, in our particular case phenomena occurring at the ion spatio-temporal scales.\\

We start with the microscopic quantities and introduce our first
parameter $\eps$, as the mass ratio of electrons and ions
$$
\eps^2\,:=\,{m_e \over m_i} \,\ll\, 1\,.
$$
Next we suppose that the temperatures of electrons and ions are of the same order $\overline{T}$ 
$$
T_i\,=\,\overline{T}\, T_i'\,, \qquad T_e\,=\,\overline{T}\, T_e'\,, {\qquad T_{\alpha \beta}\,=\,\overline{T}\, T_{\alpha \beta}'\,,}
$$
meaning that the {thermal (microscopic) speeds of the two species are widely different and scale as}
$$
\overline{v}_i\,:=\, v_{th,i}=\sqrt{ k_B {\overline T} \over m_i}\,, \qquad \overline v_e\,:=\, v_{th,e}=\sqrt{ k_B {\overline T} \over m_e}={1 \over \eps}\, \overline{v}_i\,.
$$
{Furthermore we shall assume that the electric and thermal energies are of the same order of
magnitude, permitting thus to scale the electric potential as $e\, {\overline
  \phi}\,=\,k_B\, \overline T$}. 
We also suppose that the plasma is quasineutral, that is, densities of electrons and ions are of the same order $\overline{n}$ 
$$
n_i\,=\,\overline{n}\, n_i'\,, \qquad n_e\,=\,\overline{n}\, n_e'\,,
$$
permitting thus to fix the magnitude of the ion plasma frequency $\omega_p$ and of  the Debye length $\lambda_D$ as
$$
\omega_p^{-1}\,:=\,\sqrt{\overline{n}\, e^2 \over \eps_0\, m_i}\,,\qquad\lambda_D:= \sqrt{\eps_0\, k_B\, \overline{T} \over \overline{n}\, e^2}\,, 
$$
which yields the relation $\overline v_{i} = \lambda_D \,\omega_p$.

At the microscopic level again, we fix a time-scale $\tau_c$ and a length-scale $l_c$, related to the ionic collisional process, namely
$$
\tau_c\,:=\,\tau_{ii}\,=\,\nu_{ii}^{-1} ,\qquad l_c\,:=\, {\overline v}_i\, \tau_c\,,
$$
where $\tau_{ii}$ corresponds to the elapsed time between two ionic collisions (collisional frequency $\nu_{ii}$) and $l_c$ is  the corresponding mean free path.\\

Finally, let us turn to the macroscopic quantities corresponding to the physical device. The macroscopic  space-scale ${\overline x}$ is fixed as the distance of interest and the time-scale corresponds to the observation time given by ${\overline t}=\overline x/\overline v_i$, hence we set for the macroscopic velocities $\overline u_\alpha:=\overline v_\alpha$ {as well as $\overline u_{ei} = \overline u_{ie}= \overline u_e$}. \\
To characterize the regime of interest, let us introduce now a second parameter $\tau$ as the ratio between micro and macro time-scales
$$
\tau\,:=\, {\tau_c \over \overline t}
$$
and a third parameter $\lambda$ as the ratio between micro and macro space-scales
$$
\lambda\,:=\,{\lambda_D \over \overline{x}}\,=\, {\overline{v_i} \over \overline{x}} \, { 1 \over \omega_p}\,=\, {1\over \omega_p\,\overline t}\,.
$$
Concerning the different intra- and inter-species collision
frequencies, we simply set
$$
\nu_{\alpha \beta}={\overline \nu_{\alpha \beta}}\, \nu_{\alpha \beta}'\,,\qquad \forall {\alpha\,, \beta} \in \{e,i\}\,,
$$
with the order-relations given by  \cite{PHYS1}
$$
{\overline\nu_{ie}}:{\overline\nu_{ii}}:{\overline\nu_{ee}}:{\overline\nu_{ei}}=\eps^2:\eps:1:1\,.
$$
{Finally let us also fix characteristic scales for the distribution functions and the collision operators
$$
{\overline f}_\alpha= {{\overline n} \over {\overline v}_\alpha}\,, \qquad  {\overline Q}_{\alpha \beta} = {\overline \nu}_{\alpha \beta}\,{\overline f}_\alpha\,. 
$$
The units or scales chosen here are adapted to the plasma regimes we want to study (electron Boltzmann regime). The reader not so familiar with the physics of tokamak fusion plasmas and its characteristic scales is referred to the introductory books \cite{CHENF,Ruther,PHYS1,hazeltine_meiss}}.

%%%%%%%%%%%%%%%%%%%%%

\subsection{Non-dimensional kinetic system}
\label{SEC22}

%%%%%%%%%%%%%%%%%%%%%

Let us observe that we have now a set of three independent parameters
$(\eps,\tau,\lambda)$, which characterize several plasma regimes. To get the non-dimensional system, let us perform the following change of variables in the starting model  \eqref{SystemKineticDimensional} 
$$
\bfx\,=\,{\overline x}\, \bfx'\,, \qquad t\,=\,{\overline t}\, t'\,, 
$$
whereas the velocities (in the two different kinetic equations) scale differently for ions and electrons, namely
$$
\quad \bfv_i\,=\,{\overline v_i}\,\bfv'\,, \qquad  \bfv_e \,=\,{\overline v_e} \bfv'\,. 
$$
{This different scaling in the velocities is fundamental for the further study, the rescaled velocities $\bfv'$ being now of the same order for ions and electrons, fact which is a considerable advantage for numerical simulations.}
Furthermore, in \eqref{PoissonDim} we set
$$
\bfE(t,\bfx)\,=\,{\overline E}\, \bfE'(t',\bfx')\,.
$$
Altogether one obtains then the following non-dimensional system (the primes were omitted for
simplicity reasons)
\begin{equation}
\left\{
 \begin{array}{ll}
   \displaystyle \partial_t f_i \,+\,   \bfv\cdot \nabla_\bfx f_i \,+\,  \bfE \cdot \nabla_\bfv f_i \,=\,  \frac{1}{\tau}\,\left(\cQ_{ii}(f_i,f_i) \,+\, \varepsilon\, \cQ_{ie}(f_i,f_e)\right)\,, 
   \\[1.1em]      
      \displaystyle \partial_t f_e \,+\, \frac{1}{\varepsilon}\, \bfv\cdot\nabla_\bfx f_e \,-\,  \frac{1}{\varepsilon}\, \bfE\cdot\nabla_\bfv f_e \,=\, \frac{1}{\tau\,\varepsilon} \left( \cQ_{ee}(f_e,f_e) \,+\, \cQ_{ei}(f_e,f_i) \right)\,,
    \end{array}
  \right.
\label{SystemKineticNonDimensional_bis}  
\end{equation}
supplemented with Poisson's equation
$$
-\lambda^2\,\Delta \phi \,=\, n_i - n_e\,, \qquad \bfE = - \nabla_\bfx \phi.
$$
\cla{Starting from this non-dimensional model, we choose the following regime:
\begin{itemize}
\item the ratio between micro and macro time-scales is considered fixed $\tau=1$; 
\item the ratio between micro and macro space-scales is considered also fixed $\lambda=1$; 
\item the electron-to-ion mass ratio $\eps \ll 1$, will be the only perturbation parameter we shall take into account.
\end{itemize}}
%$\lambda$ and $\tau$ are fixed  ($\lambda=\tau=1$) and keep only the perturbation parameter $\eps\ll 1$. 
This choice permits to focus on the electron adiabatic  asymptotics, without adding additional difficulties coming from the quasi-neutral limit $\lambda \ll 1$  studied for instance in \cite{Bob2, HKR}. The fact that we set $\tau =  1$ is justified by our aim to keep the ions kinetic. Other asymptotic regimes can be naturally investigated.\\

\noindent The rescaled macroscopic quantities are given now  for $\alpha\in\{e,\,i\}$ by
\begin{equation*}
  \label{MMo_1D}
  \left\{
  \begin{array}{l} 
 \ n_{\alpha} := \ds\int_{\RR^d} f_{\alpha} \,\dD\bfv \,,%\label{MMo1}
 \\[1.1em]
n_{\alpha}\, \bfu_{\alpha} :=\ds \int_{\RR^d}  \bfv \,f_{\alpha}\, \dD\bfv \,,%\label{MMo2}
 \\[1.1em]
w_{\alpha} := \ds\frac{1}{2}\int_{\RR^d} |\bfv|^2 \, f_{\alpha}\,\dD \bfv \,=\, \frac{d}{2} \, n_{\alpha} T_{\alpha} + {1\over 2} \,  n_{\alpha}\,|\bfu_{\alpha}|^2\,,%\label{MMo3}
 \end{array}\right.
\end{equation*}
where
$$
d\, n_{\alpha}\, T_{\alpha}:= \int_{\RR} |\bfv -  \bfu_{\alpha}|^2\, f_{\alpha}\,\dD\bfv\,,
$$
and the pressure tensor $\bfp_\alpha$ as well as the heat flux $\bfq_\alpha$ are given by
$$
\left\{
\begin{array}{l}
  \ds\bfp_{\alpha} \,:=\, \int_{\RR^d}  (\bfv - \bfu_{\alpha})\otimes(\bfv - \bfu_{\alpha})\,f_{\alpha}\,\dD \bfv \,,
  \\[1.1em]
  \ds\bfq_{\alpha} \,:=\, \frac{1}{2}\int_{\RR^d}  (\bfv - \bfu_{\alpha})\,|\bfv - \bfu_{\alpha}|^2 \,f_{\alpha}\,\dD \bfv \,,%\label{MMo5}
  \end{array}\right.
$$
whereas the non-dimensional collision operators read now for $\alpha\in\{e,\,i\}$ as
$$
  \cQ_{\alpha\alpha}(f_\alpha,f_\alpha) \,=\, \nu_{\alpha\alpha}\, \Div_{\bfv}\left( (\bfv -  {\bfu_\alpha})\, {f_\alpha} \,+\, {T_\alpha}\, \nabla_{\bfv} f_\alpha \right)\,,
$$
and
$$
\left\{
\begin{array}{l}
  \ds \cQ_{ei}(f_e,f_i) \,=\, \nu_{ei}\,\Div_{\bfv}\left( ({\bfv} -  {\bfu_{ei}}) \,{f_e} \,+\, T_{ei}\, \nabla_{{\bfv}} {f_e} \right)\,,
  \\[1.1em]
\ds \cQ_{ie}(f_i,f_e) \,=\,  \nu_{ie}\,\Div_{{\bfv}}\left( ({\bfv} - {\bfu_{ie}\over\eps})\, {f_i} \,+\, T_{ie}\, \nabla_{{\bfv}} {f_i} \right)\,,
\end{array}\right.
$$
with the mixed quantities
\be \label{defuieAdim}
{\bfu_{ei}} \,=\, \frac{{\bfu_e} \,+\, \eps\, {\bfu_i}}{2} \,=\, {\bfu_{ie}}\,,
\ee
\be
T_{ei} \,=\,  T_{ie} \,=\, \frac{1}{1 \,+\, \varepsilon^2} \left( {T_e} \,+\, \varepsilon^2\, {T_i} \,+\, {|{\bfu_e} \,-\, \eps\, {\bfu_i}|^2 \over 2d} \right)\,.%=T_e+\eps^2(T_i-T_e)+ {({u_e} - \eps {u_i})^2 \over 2}+{\mathcal O}(\eps^4)\, .
\label{defThetaeiieAdim}
\ee
{To give only an example for these scalings, let us detail the temperature rescaling. Starting from \eqref{mixed2} and using the characteristic values defined in Section \ref{SEC_Ca} one obtains
$$
T_{ei} = \overline{T}\, {T_e'+{m_e/m_i}\, T_i' \over m_e/m_i +1} + {m_e \, {\overline u_e}^2 \over 1+ m_e/m_i } \, { | u_e' - {\overline{u_i} \over \overline{u_e}}\, u_i' |^2\over 2d k_B}=  \frac{\overline T}{1 \,+\, \varepsilon^2} \left( {T_e'} \,+\, \varepsilon^2\, {T_i'} \,+\, {|{\bfu'_e} \,-\, \eps\, {\bfu'_i}|^2 \over 2d} \right)\,.
$$}

The non-dimensional model \eqref{SystemKineticNonDimensional_bis}-\eqref{defThetaeiieAdim} will be our starting point for the $\eps \rightarrow 0$ asymptotic study. One can observe that the time scale of interest in this paper corresponds to the average time between two ionic collisions. This time is much larger than the characteristic time of electron collisions. As a consequence, we can expect that in the limit $\eps \rightarrow 0$ the ions remain kinetic and the electrons reach a certain macroscopic regime due to the numerous collisions they undertake.\\

Let us now prove the properties of conservation and entropy decay, already presented for the dimensional operators. Firstly, let us introduce for $\alpha\in\{e,\,i\}$ the adimensional Maxwellian distributions $\cM_{\alpha}$ {obtained by rescaling \eqref{localM}, {\it i.e.}}
$$
\cM_{\alpha}(t,x,v) \,:=\, \frac{n_\alpha}{\left(2 \pi\, T_{\alpha}\right)^{d/2}}\, \exp\left(-\frac{|\bfv - \bfu_{\alpha}|^2}{2\, T_{\alpha}}\right),\qquad \forall \alpha\in\{e,\,i\}\,,
$$
which correspond to the equilibrium distributions of the operators $\cQ_{ee}$ and $\cQ_{ii}$. Then we also define the equilibria for $\cQ_{ei}$ and  $\cQ_{ie}$ via
$$
\left\{
\begin{array}{l}
 \ds \cM_{ei}(t,x,v) \,:=\, \frac{n_e}{\left(2 \pi\, T_{ei}\right)^{d/2}}\, \exp\left(-\frac{|\bfv - \bfu_{ei}|^2}{2\, T_{ei}}\right),
    \\[1.1em]
\ds\cM_{ie}(t,x,v) \,:=\, \frac{n_i}{\left(2 \pi\, T_{ie}\right)^{d/2}}\, \exp\left(-\frac{|\eps\bfv - \bfu_{ie}|^2}{2\,\eps^2 \,T_{ie}}\right),
\end{array}
\right.
$$
{which are nothing but the rescaled versions of \eqref{oufff}}.
To simplify the formulae, let us denote in the following by $h_{\alpha\beta}$ respectively $h_{\alpha}$ the functions 
\be
\label{def:h}
h_{\alpha\beta} := \frac{f_\alpha}{\cM_{\alpha\beta}}\,, \qquad h_\alpha := h_{\alpha\alpha}\,.
\ee
With these new notations, we can rewrite the collision operators in the simpler form
$$
\cQ_{\alpha\beta}(f_\alpha,f_\beta) \,=\, \nu_{\alpha\beta}\,\Div_{\bfv}\left( T_{\alpha\beta}\, \cM_{\alpha\beta}\, \nabla_{{\bfv}} {h_{\alpha\beta}} \right)\,,\qquad \forall \alpha,\beta\in\{e,\,i\}\,.
$$

\begin{prop}
  \label{prop:1}
Under the constraint
\be
\nu_{ei} \,n_e \,\,=\,\, \nu_{ie} \,n_i\,,
\label{relationCollisionFrequencyAdimentional}
\ee
 corresponding to the {rescaled version of }
 \eqref{relationCollisionFrequency}, we have the following conservations
\begin{equation*}
  \int_{\mathbb{R}^d} \cQ_{ei}(f_e,f_i)
\begin{pmatrix}
 \bfv \\[0.7em]
 \ds\frac{|\bfv|^2}{2}
\end{pmatrix}
\dD \bfv\,
\,+\, \int_{\mathbb{R}^d} \cQ_{ie}(f_i,f_e)
\begin{pmatrix}
\eps\,\bfv \\[0.7em]
\ds \eps^2\,\frac{|\bfv|^2}{2}
\end{pmatrix}
\dD \bfv\,=\,0\,.
\end{equation*}
Furthermore, defining the inter-species entropy dissipation  $\cI$ by
\begin{equation}
\cI(t,x) \;:=\,  -\int_{\mathbb{R}^d} \left[\cQ_{ei}(f_e,f_i) \ln(f_e) \,+\, \eps^2 \,\cQ_{ie}(f_i,f_e) \ln(f_i)\right]\,\dD\bfv\,,
\label{EntropyDecayNonDim}
\end{equation}
we have
$$
\cI \,=\, \int_{\mathbb{R}^d}\left[\nu_{ei}\,T_{ei}\frac{\cM_{ei}}{h_{ei}} \,\left|\nabla_vh_{ei}\right|^2\,+\,\eps^2\,\nu_{ie}\,T_{ie}\frac{\cM_{ie}}{h_{ie}} \left|\nabla_v h_{ie}\right|^2\right]\,\dD v  \geq 0\,.
$$
\end{prop}

\begin{proof}
Conservations of mass, momentum and energy follow from direct computations, whereas entropy dissipation is obtained observing simply that
\be \label{ddd}
\int_{\mathbb{R}^d} \cQ_{ei}(f_e,f_i) \,\ln(\cM_{ei})\, \dD\bfv  \,+\, \eps^2\int_{\mathbb{R}^d} \cQ_{ie}(f_i,f_e) \,\ln(\cM_{ie}) \,\dD\bfv  \,=\, 0.
\ee
\end{proof}

For the investigation of the asymptotic limit $\eps \rightarrow 0$ it will be necessary to have in mind the rescaled macroscopic electron equations corresponding to \eqref{SystemKineticNonDimensional_bis}, namely 
{\be \label{fluid}
\left\{
\begin{array}{l}
 \ds \eps \, \pa_t n_e \,+\, \Div_\bfx (n_e \, \bfu_e) \,=\, 0 \,,
 \\[1.1em]
 \ds\eps\, \pa_t (n_e\,\bfu_e) \,+\,  \Div_\bfx \left(n_e\, \bfu_e\otimes \bfu_e \,+\,\bfp_e \right)\,+\, n_e\, \bfE\,=\,-\nu_{ei}\, n_e \,  {\bfu_e\,-\, \eps\,\bfu_i \over 2}\,,
 \\[1.1em]
 \ds\eps\, \pa_t w_e \,+\, \Div_\bfx \left(w_e \,\bfu_e\, +\,
  \bfp_e \, \bfu_e + \bfq_e\right) \,+\, n_e\,\bfu_e \cdot \bfE\,=\, S_{ei}\,,
\end{array}\right.
\ee
where {the energy exchange term reads}
\be
\label{def:Sei}
 \ds S_{ei}\, :=\, \ds \int_{\mathbb{R}^d} \cQ_{ei}(f_e,f_i)\,\frac{|\bfv|^2}{2} \dD\bfv
  \, =\ - \nu_{ei}\, n_e \,\left[ d\, (T_e- T_{ei}) \,+\, \bfu_e\cdot\frac{\bfu_e-\eps\,\bfu_i}{2}\right]\,,
\ee
{consisting of a first term corresponding to the temperature equilibration and a second term corresponding to the work done by friction.}
Observe that if we assume that all macroscopic quantities are uniformly bounded with respect to $\eps$, {and replacing $T_{ei}$ by the expression \eqref{defThetaeiieAdim}, we obtain}
$$
 S_{ei}  \,=\,  \nu_{ei}\, n_e \,\left( -{\eps  \over 2} \, {\bfu_i \cdot \bfu_e} \,+\,  \eps^2\,d\, (T_i - T_e) 
 \,+\, \eps^2\, {|\bfu_i|^2- |\bfu_e|^2 \over 2}\right) \,+\, {\mathcal O}(\eps^3)\,,
$$
{which permits to see that temperature equilibration between ions and electrons occurs on a long time scale when $\eps$ is small (factor $\eps^2$) which means that ions and electrons can become Maxwellians (due to the collisions) long before their temperature equilibrate.}
This system is not closed as the pressure tensor $\bfp_e$ and the heat flux $\bfq_e$ cannot be expressed with the help of the other three macroscopic variables $(n_e,\bfu_e,w_e)$. However in the limit $\eps \rightarrow 0$ one can close this system, as shall be shown in the sequel, the asymptotic model being given in Theorem \ref{th:1}.\\

%%%%%%%%%%%%%%%%%%
\section{Formal derivation of the asymptotic model}
\label{SEC3}
\setcounter{equation}{0}
\setcounter{figure}{0}
\setcounter{table}{0}

%%%%%%%%%%%%%%%%%% 
Let us consider from now the one-dimensional case.
The main goal of this section is to understand more about the asymptotic
limit $\eps \rightarrow 0$ of the following kinetic system
\begin{equation}
\left\{
 \begin{array}{ll}
      \displaystyle \partial_t f_i^\eps +   v\, \partial_x f_i^\eps  +  E^\eps  \, \partial_v f_i^\eps  =    \cQ_{ii}(f_i^\eps ,f_i^\eps ) + \varepsilon\, \cQ_{ie}(f_i^\eps ,f_e^\eps )\,, \\[1.1em]      
      \displaystyle \eps\,\partial_t f_e^\eps  + \, v\, \partial_x f_e^\eps  -  \, E^\eps \, \partial_v f_e^\eps  =  \cQ_{ee}(f_e^\eps ,f_e^\eps ) + \cQ_{ei}(f_e^\eps ,f_i^\eps )\,,
    \end{array}
  \right.
\label{eq:kin}  
\end{equation}
coupled to Poisson's equation
\be \label{poisson}
- \partial_{xx} \phi^\eps  = n_i^\eps  - n_e^\eps \,, \qquad E^\eps  = - \partial_x \phi^\eps \,.
\ee
To fix the potential, let us impose the constraint of zero average
\be
\label{constr-phi}
\int_{\TT} \phi^\eps (t,x)\, \dD x \,=\,0\,, \quad \forall t\,>\,0\,.
\ee
The study of the asymptotic $\varepsilon\rightarrow 0$ requires estimates that are uniform with respect to $\varepsilon$. 
For our coupled system, the only natural identities providing such
bounds are mass conservation, free energy and entropy
inequalities. Let us first introduce the  kinetic energy associated with each species
\[
K(t) \,:=\, K_e(t) \,+\, \,K_i(t)\,,
\]
where 
$$
% \be
%\label{Kkin}
K_{\alpha} \,:=\, \frac{1}{2}\,\iint_{\TT\times\mathbb{R}}|v|^2\,f_{\alpha}^\eps  \,\dD v \dD x\,, \qquad\forall\alpha\in\{e,\,i\}\,.
% \ee
$$
The characteristic energy related to the electrostatic effects is the electric energy and reads
%\be
%\label{Upot}
$$
U(t) \,:=\, \frac{1}{2}\,\int_{\TT}|\partial_x\phi^\eps |^2 \dD x \,=\, \frac{1}{2}\,\int_{\TT}\phi^\eps  \,(n_i^\eps  \,-\, n_e^\eps ) \,\dD x\,,
%\ee
$$
where the last equality stems from Poisson's equation. With these notation we have now the following result.
\begin{prop}
  \label{prop:2} {\bf (Energy conservation)}
  Suppose that $(f_e^\eps , f_i^\eps , \phi^\eps )$ is a   solution of \eqref{eq:kin}-\eqref{constr-phi} such that $f_e^\eps $ and $f_i^\eps $ are nonnegative and satisfy initially
\be
  \label{init:0}
  \iint_{\TT \times \RR} \left[f_e^\eps (0) + f_i^\eps (0) \right] \,(1\,+\,|v|^2) \,\dD v\,\dD x  < \infty\,.
\ee
  Then, one has the energy conservation, for all $\varepsilon>0$ 
     \[
 \mathcal{E}(t) \,:=\, U(t) + K(t) \,=\,{\mathfrak E}, \qquad \forall \,t\geq 0\,,
 \]
 where ${\mathfrak E}$ is given by the initial value
 \be
 \label{init:ener}
{\mathfrak E}\,:=\,\mathcal{E}(0) \,<\,\infty\,.
 \ee
\end{prop}
\begin{proof}
Let us first remark that the condition \eqref{init:0} together with \eqref{poisson} implies $||\partial_x \phi^\eps  (0)||^2_{L^2(\TT)} < \infty$, such that we have indeed a bounded initial energy $\mathcal{E}(0) \,<\,\infty$.

Now, let us multiply the first equation in \eqref{eq:kin} by $\eps\, v^2/2$, the second one by $v^2/2$ and integrate with respect to $(x,v)\in\TT \times \mathbb R$. This yields after summing up the two equations and integrating by parts
  \begin{eqnarray}
  \label{eq:tmp1}
  &&\eps\,\frac{\dD K}{\dD t}(t) - \int_{\TT} E^\eps  \,(\eps\,n_i^\eps \,u_i^\eps  - n_e^\eps  \,u_e^\eps )\,\dD x
  \\
  \nonumber
  &&= \frac{1}{2}\iint_{\TT\times\mathbb R} |v|^2\left(\varepsilon^2\, \cQ_{ie}(f_i^\eps ,f_e^\eps ) + \cQ_{ei}(f_e^\eps ,f_i^\eps )\right) \dD v\, \dD x\,. 
  \end{eqnarray}
On one hand, applying Proposition \ref{prop:1}, we show that the right hand side vanishes. On the other hand from the continuity equation
$$
\eps\,\partial_t ( n_i^\eps  - n_e^\eps ) \,+\,\partial_x(\eps \,n_i^\eps  \,u_i^\eps  \,-\, n_e^\eps \, u_e^\eps ) \,=\, 0\,,  
$$
and Poisson's equation for $\phi^\eps $, we get that the second term in \eqref{eq:tmp1} becomes
$$
-\int_{\TT} E^\eps \, (\eps\,n_i^\eps \,u_i^\eps  \,-\, n_e^\eps \, u_e^\eps )\,\dD x \,=\,  \eps\,\frac{\dD U}{\dD t}(t)\,.
$$
Finally gathering the latter equalities, we get the conservation of the total energy
\[
 \eps\,\frac{\dD\mathcal{E}}{\dD t}(t) \,=\, \eps\,\frac{\dD}{\dD t}\left(U(t) \,+\, K(t)\right) \,=\,0, \qquad \forall \,t\,\geq\, 0\,.
 \]
\end{proof}
Let us define now the entropy of each species by the formula
\[
\cH_{\alpha}(t) \,:=\, \iint_{\TT\times\mathbb{R}} f_{\alpha}^\eps \ln(f_{\alpha}^\eps )\, \dD v\, \dD x\,, \qquad\forall \,\alpha\in\{e,\,i\}\,,
\]
and prove the following result.
\begin{prop}{\bf (Entropy decay)}
   \label{prop:3}
   Suppose that $(f_e^\eps , f_i^\eps , \phi^\eps )$ is a  solution of \eqref{eq:kin}-\eqref{constr-phi}, such that $f_e^\eps $ and $f_i^\eps $ are nonnegative, satisfying moreover
\be
  \label{init:1}
  \cH_e(0) + \cH_i(0)  < \infty\,.
\ee
   Then, one has the following entropy estimate, for all $\varepsilon>0$
  \begin{eqnarray*}
    \cH_e(t) \,+\, \cH_i(t)  &+&\frac{1}{\eps}\int_0^t
                                           \iint_{\TT\times\mathbb R}
                                           \left[ \nu_{ee}\,
                                           T_{ee}^\eps \, {\cM_{e}
                                           \over h_e^\eps }\,
                                           \left|\partial_vh_e^\eps \right|^2
                                           \,+\,\nu_{ei}\,T_{ei}^\eps \,\frac{\cM_{ei}}{h_{ei}^\eps }
                                           \,\left|\partial_vh_{ei}^\eps \right|^2
                                           \right]\dD v\,\dD x \, \dD s
    \\[1.1em]
    &+&\int_0^t\iint_{\TT\times\mathbb R} \left[ \nu_{ii}\, T_{ii}^\eps \,
        {\cM_{i} \over h_i^\eps }
        \left|\partial_vh_i^\eps \right|^2\,+\,
        \eps\,\nu_{ie}\,T_{ie}^\eps \,\frac{\cM_{ie}}{h_{ie}^\eps }
        \left|\partial_vh_{ie}^\eps \right|^2\right]\dD v \dD x \dD s
        \\[1.1em]
    &&\leq \cH_e(0) + \cH_i(0),\qquad \forall t \ge 0\,,
  \end{eqnarray*}
  where $h_{ei}^\eps $ and $h_e^\eps $ are given in \eqref{def:h}.
\end{prop}
\begin{proof}
The entropy estimates are obtained by multiplying the first two equations of \eqref{eq:kin} by $\eps(\ln(f_i^\eps ) + 1) $ and $\ln(f_e^\eps ) + 1$ respectively and integrating in $(x,v)$. This yields
\begin{eqnarray*}
\begin{array}{lll}
 \ds  \eps\,\frac{\dD }{\dD t}\left(\cH_e(t) \,+\, \cH_i(t) \right)&=&\ds -\nu_{ee} \iint_{\TT\times\mathbb R} T_{ee}^\eps \, {\cM_{e} \over h_e^\eps } \,\left|\partial_vh_e^\eps \right|^2\dD v\,\dD x \\[1.1em]
&&\ds -\,\eps\,\nu_{ii}\iint_{\TT\times\mathbb R} T_{ii}^\eps \, {\cM_{i} \over h_i^\eps } \,\left|\partial_vh_i^\eps \right|^2\,\dD v\, \dD x\,-\, \int_{\TT} \cI(t,x)\, dx\,,
\end{array}
\end{eqnarray*}
where $\cI$ is given in \eqref{EntropyDecayNonDim}. {For these computations we needed again \eqref{ddd} as well as the conservation laws.}
Dividing by $\eps$ and using the expression of $\cI$, permits to get
\begin{eqnarray*}
 \frac{\dD }{\dD t}\left(\cH_e(t) \,+\, \cH_i(t) \right) =&& -\frac{1}{\eps}\iint_{\TT\times\mathbb R} \left[ \nu_{ee} \, T_{ee}^\eps \, {\cM_{e} \over h_e^\eps }  \left|\partial_vh_e^\eps \right|^2 \,+\,\nu_{ei}\,T_{ei}^\eps \,\frac{\cM_{ei}}{h_{ei}^\eps }  \left|\partial_vh_{ei}^\eps \right|^2 \right]\dD v\,\dD x
  \\ &&-\, \iint_{\TT\times\mathbb R} \left[ \nu_{ii}\, T_{ii}^\eps \, {\cM_{i} \over h_i^\eps } \left|\partial_vh_i^\eps \right|^2\,+\, \eps\,\nu_{ie}\,T_{ie}^\eps \,\frac{\cM_{ie}}{h_{ie}^\eps } \left|\partial_vh_{ie}^\eps \right|^2\right]\dD v \,\dD x\,.
\end{eqnarray*}
Integrating over $[0,t]$ and keeping only the two dominant terms on the right hand side, yields the entropy estimate with the corresponding dissipation.
\end{proof}

Now, let us formally derive the asymptotic model when $\eps\rightarrow 0$.

\begin{theorem}{\bf(Asymptotic limit)}
  Suppose that for each $\eps >0$, $(f^\eps_i,f^\eps_e,\phi^\eps)$ is
  a  solution to the coupled system \eqref{eq:kin}-\eqref{constr-phi} satisfying \eqref{relationCollisionFrequencyAdimentional}, \eqref{init:0}-\eqref{init:ener}, \eqref{init:1} and
  $$
  \iint_{\mathbb{T} \times\mathbb{R}} f_e^\eps(t=0,x,v)\, \dD v \,\dD x\,=\,\iint_{\mathbb{T} \times\mathbb{R}} f_i^\eps(t=0,x,v)\, \dD v \,\dD x\,=\,{\mathfrak N}\,,
  $$
  for some fixed ${\mathfrak N}>0$. Furthermore, for any final time
  $T_{\rm end}>0$, let us assume that the family defined by the
  macroscopic quantities
  $\bfU^\eps_\alpha=(n^\eps_\alpha,u^\eps_\alpha,T^\eps_\alpha)$ for
  $\alpha\in\{e,\, i\}$ are relatively compact in $L^1(0,T_{\rm end},
  L^\infty(\TT))$.  Then, when $\eps$ tends towards zero, the
  distribution function $f_e^\eps$ tends towards a local Maxwellian $f_e^0$ of
  the form (Maxwell-Boltzmann distribution) 
  \be
  \label{def:MB}
  \left\{
    \begin{array}{l}
\ds f_e^0(t,x,v) \,=\,  { n_e^0(t,x) \over \sqrt{2\, \pi\, T_e^0(t)}} \,
      \exp\left(-\frac{v^2}{2 T_e^0(t)}\right)\,,
      \\
      \ds n_e^0(t,x)\,=\,c^0(t)\,
      \exp\left(\frac{\phi^0(t,x)}{T_e^0(t)}\right)\,,
      \end{array}\right.
\ee
where $c^0(t)$ is given such that
\be
\label{def:ct}
\int_{\TT} n_e^0(t,x) \,\dD x \,=\, {\mathfrak N},\qquad \forall \;t \geq 0\,,
\ee
whereas $(\phi^0, T_e^0)$ is such that $T_e^0(t)$ only depends on the time-variable, and is computed via the energy conservation
\be
\label{def:Te}
\frac{{\mathfrak N}}{2}\, T_e^0(t) \,+\,\frac{1}{2}\,\int_{\TT}
|\partial_x \phi^0(t,x)|^2 \dD x \,+\, \frac{1}{2}\,\iint_{\TT\times\RR} f_i^0(t,x,v) \,v^2\dD v \dD x
=\, {\mathfrak E},\qquad \forall t\geq 0
\ee
and $\phi^0$ is solution to {the nonlinear Poisson-Boltzmann equation}
\be
\label{def:phi}
  - \partial_{xx} \phi^0 \,+\, c^0(t)\,\exp\left( \frac{\phi^0}{T_e^0(t)}\right)\,=\, n_i^0\,, \qquad  E^0\,=\,-\partial_x \phi^0\,,
        \ee
supplied with the additional constraint
  \be
\label{cstr}
\int_{\TT}   \phi^0(t,x)\, \dD x \,=\,0\,, \qquad \forall t\ge0\,.
\ee
The ion distribution function $f_i^0$ satisfies then the
Vlasov-Fokker-Planck equation
\be
\left\{
 \begin{array}{ll}
      \displaystyle \partial_t f_i^0 \,+\,  v \, \partial_x f_i^0 \,+\,
   E^0\,  \partial_v f_i^0 \,=\,   \cQ_{ii}(f_i^0,f_i^0)\,,
   \\[1.1em]
  \ds n_i^0 = \int_\RR f_i^0 \dD v\,. 
     \end{array}
  \right.
  \label{eq:lim2}
  \ee
\label{th:1} 
\end{theorem}
\begin{proof}
On one hand, thanks to the uniform bounds established in Proposition \ref{prop:2} and \ref{prop:3} on the total energy and the entropy, one can show that up to a subsequence we have for $\alpha\in\{e,\,i\}$
  $$
f^\eps_\alpha \rightharpoonup f_\alpha^0 \quad\textrm{ weakly-$\star$ in }\,  L^{\infty}\left(0,T_{\rm end}, L^1(\TT\times\RR)\right),\quad\textrm{as}\,\,\, \eps\rightarrow 0\,.
  $$
On the other hand, from the assumption on the densities $n_e^\eps$ and
$n_i^\eps$ and using Poisson's equation \eqref{poisson}, we can show that up to
a subsequence,  $E^\eps$ converges to $E^0$ strongly in $L^1(0,T_{\rm
  end},L^\infty(\TT))$. Together with the weak-$*$ convergence of the
distribution function $(f^\eps_i)_{\eps>0}$ and $(f^\eps_e)_{\eps>0}$
in $L^{\infty}\left(0,T_{\rm end}, L^1(\TT\times\RR)\right)$, this
yields for $\alpha\in\{i,\,e\}$,  and for any test function $\varphi\in
\cC^\infty_c([0,T_{\rm end})\times\TT\times \RR)$,
  $$
\int_0^{T_{\rm end}}\iint_{\TT\times\RR} E^\eps\, f_\alpha^{\eps}
\,\partial_v\varphi \,\dD x\,\dD v \,\dD t \,\rightarrow\, \int_0^{T_{\rm
    end}}\iint_{\TT\times\RR}  E^0\,f_\alpha^0 \,\partial_v\varphi \,\dD
x\,\dD v \,\dD t, \quad\textrm{ as }\,\eps\rightarrow 0\,.
$$
Furthermore, using that $(n^\eps_\alpha,u^\eps_\alpha,T^\eps_\alpha)_{\eps>0}$ is
relatively compact in $L^1(0,T_{\rm end},L^\infty(\TT))$, we also get
that $\cQ_{\alpha\beta}(f_\alpha^\eps,f_\beta^\eps)$ weakly converges
to  $\cQ_{\alpha\beta}(f_\alpha^0,f_\beta^0)$. Altogether we have thus that the limit
$(f_e^0,f_i^0)$ is a solution to the following system
\begin{equation}
  \label{tmp2}
\left\{
 \begin{array}{ll}
      \displaystyle \partial_t f_i^0 +   v\, \partial_x f_i^0 +  E^0 \, \partial_v f_i^0 \,=\,    \cQ_{ii}(f_i^0,f_i^0) - \nu_{ie} {u_e^0 \over 2} \,\partial_v f_i^0\,, \\[1.1em]      
      \displaystyle v\, \partial_x f_e^0 -  \, E^0\, \partial_v f_e^0 =  \cQ_{ee}(f_e^0,f_e^0) + \cQ_{ei}(f_e^0,f_i^0)\,,
    \end{array}
  \right.  
\end{equation}
coupled with Poisson's equation for the potential $\phi^0$
$$
- \partial_{xx} \phi^0 \,=\, n_i^0 \,-\, n_e^0\,, \qquad E^0 \,=\, - \partial_x \phi^0\,,
$$
with the constraint of zero average
$$
\int_{\TT} \phi^0(t,x)\, \dD x \,=\,0\,, \quad \forall t\,>\,0\,.
$$
It remains to check that in the $\eps \rightarrow 0$ limit $u_e^0=0$ and $f_e^0$ is of the form of a local Maxwellian. Indeed, in view of the entropy dissipation given in Proposition \ref{prop:3} and the conservation of mass, we obtain
$$
f_e^0 \,=\,\cM_{n_e^0,u_e^0,T_e^0}\,\quad\textrm{and}\quad  f_e^0 \;=\,\cM_{n_e^0,u_{ei}^0,T_{ei}^0}\,,
$$
where $u_{ei}^0=u_e^0/2$ and $T_{ei}^0=T_e^0 + |u_e^0|^2/2$ are obtained by passing to the limit in \eqref{defuieAdim}-\eqref{defThetaeiieAdim}. 
Hence this yields that $u_e^0=0$ and  that $f_e^0$ is a local Maxwellian of the form
$$
f_e^0 (t,x,v)=\cM_{n_e^0,0,T_e^0}\,,
$$
whereas the unknowns $(n_e^0,T_e^0)$ are still to be determined.
Finally, inserting now the Maxwellian $f_e^0=\cM_{n_e^0,0,T_e^0}$ in the second equation of \eqref{tmp2}, yields
$$
v\, \partial_x \cM_{n_e^0,0,T_e^0} \,-\, E^0\, \partial_v \cM_{n_e^0,0,T_e^0} =0\,,
$$
or equivalently
$$
v \left[ {\partial_x n_e^0 \over n_e^0} \,+\, \left( {v^2 \over 2\, T_e^0} -{1 \over 2} \right){\partial_x T_e^0 \over T_e^0}  \right]  \,-\, v\, {\partial_x \phi^0 \over T_e^0}\,=\,0\,.
$$
This permits to get the asymptotic model \eqref{def:MB}-\eqref{eq:lim2}, by comparing the terms of the same order in $v$.
\end{proof}

It is for the moment not clear if it is possible to perform this
formal proof even at the fluid level, meaning starting from the
electron moment equations \eqref{fluid}. It seems that only the
kinetic framework, especially the powerful H-theorem, permits to
obtain that in the limit the mean electron velocity vanishes $u_e
\equiv 0$ and that the temperature $T_e^0(t)$ is only time-dependent.
\\
{\begin{remark} \label{remark_L}
The Limit-model \eqref{def:MB}-\eqref{eq:lim2} has an equivalent formulation given by
\be \label{L_bis}
(L)'\,\,\, \left\{
\begin{array}{l}
  \displaystyle \partial_t f_i^0 \,+\,  v \, \partial_x f_i^0 \,+\,
   E^0\,  \partial_v f_i^0 \,=\,  \nu_{ii}\, \partial_v \left[ (v-u_i^0)\, f_i^0+ T_i^0\, \partial_v f_i^0\right]\,,\\[3mm]
     \displaystyle   v \, \partial_x f_e^0 \,-\,
   E^0\,  \partial_v f_e^0 \,=\,  (\nu_{ee} + \nu_{ei}) \, \partial_v \left[ v\, f_e^0+ T_e^0\, \partial_v f_e^0\right]\,,\\[3mm]
  \displaystyle   - \partial_{xx} \phi^0 \,=\, n_i^0-n_e^0\,, \qquad  E^0\,=\,-\partial_x \phi^0\,,\\[3mm]
   \displaystyle \int_{\TT} n_e^0(t,x) \,\dD x \,=\, \int_{\TT} n_i^0(t,x) \,\dD x \,=\,{\mathfrak N},\qquad \forall \;t \geq 0\,, \\[3mm]
   \displaystyle \frac{1}{2}\,\iint_{\TT\times\RR} f_e^0(t,x,v) \,v^2\dD v \dD x \,+\,\frac{1}{2}\,\int_{\TT}
|\partial_x \phi^0(t,x)|^2 \dD x \,+\, \frac{1}{2}\,\iint_{\TT\times\RR} f_i^0(t,x,v) \,v^2\dD v \dD x
=\, {\mathfrak E}\,.

\end{array}
\right.
\ee
\end{remark}}
\vspace{0.2cm}

The next theorem certifies the existence and uniqueness of a solution of the just obtained asymptotic model \eqref{def:MB}-\eqref{eq:lim2}, for a given ion distribution function $f_i^0$.
\begin{theorem} {\bf{(Well posedness of the asymptotic model)}}
  Let us fix the total number of  electrons ${\mathfrak N}>0$ and the initial energy ${\mathfrak E}>0$. Furthermore, assume that the ion distribution function $f_i$ is known, sufficiently smooth and such that $(n_i,w_i) \in L^\infty(\RR^+;L^1(\mathbb{T};\RR^+))^2$, satisfying
  $$
  \int_\mathbb{T} n_i \,\dD x\,=\;{\mathfrak N}\,.
  $$
  Then there exists a unique solution $(\phi,T_e) \in
  L^\infty(\RR^+;H^1(\mathbb{T})) \times L^\infty(\RR^+)$ to the
  non-linear elliptic problem \eqref{def:ct}-\eqref{cstr}.
Moreover, if $n_i \in L^\infty(\RR^+;L^2(\mathbb{T}))$ one gets  $\phi \in L^\infty(\RR^+;H^2(\mathbb{T}))$.
\end{theorem}
\begin{proof}
Let us  first introduce the space 
$$
{\mathcal H}\,:=\, \{\, g\in H^1(\mathbb{T}) \,\, / \,\,  \int_\mathbb{T}
g(x)\, \dD x \,=\,0\,\}\,.
$$
Observe that in the limit problem \eqref{def:ct}-\eqref{cstr}, the time $t
\in \RR^+$ is simply a parameter, hence we fix and
drop it in the sequel.  
The proof of this theorem is based on the construction and study of the map ${\mathcal E}:\RR^+ \rightarrow \RR^+$, given for any $T \in \RR^+$ by
$$
{\mathcal E}(T):={{\mathfrak N} \over 2}\, T \,+\, {1 \over 2}\,
\int_\mathbb{T}  \left| \partial_x \phi_T(t,\cdot) \right|^2 \dD x
\,+\, \int_\mathbb{T} w_i (t,\cdot) \,\dD x\,,
$$
where $\phi_T$ is the solution to  non-linear elliptic equation
\eqref{def:phi} with conditions \eqref{def:ct}
and \eqref{cstr}, and with $T_e \,\equiv\, T>0$.

The aim is now to show that there exists $T_\star >0$ such that ${\mathcal
  E}(T_\star)={\mathfrak E}$. This shall be done in several steps: first we prove that for any $T>0$
there exists a smooth potential, denoted $\phi_T$, to \eqref{def:phi}
with conditions \eqref{def:ct} and \eqref{cstr}, hence that ${\mathcal
  E}(T)$ is a well-defined mapping; secondly we shall prove that this
mapping ${\mathcal E}$ is continuous, nondecreasing with ${\mathcal
  E}(0) \leq {\mathfrak E}$ and
$$
\lim_{T\rightarrow \infty} {\mathcal
  E}(T)=\infty.
$$

\paragraph{\bf Step 1 :  Study of the map $T>0 \mapsto \phi_T \in  {\mathcal H} $ for fixed $t \in \RR^+$.}

For a given $T>0$ standard
elliptic theory permits to show the existence and uniqueness of a
solution $\phi_T \in {\mathcal H} $ to the non-linear elliptic
equation \eqref{def:phi} with conditions
\eqref{def:ct} and \eqref{cstr}. This is based on the minimization of the strictly convexe, differentiable functional on the space ${\mathcal H}$
$$
{\mathcal L}(\phi)\,:=\,{1 \over 2}\, \int_\mathbb{T}  \left|
  \partial_x \phi \right|^2 \dD x \,+\, {\mathfrak N}\, T\,  \ln
\left(  \int_\mathbb{T} e^{\phi/T}\, \dD x \right) - \int_\mathbb{T}
n_i\, \phi\, \dD x\,.
$$
Remark that one has the compact injection $H^1(\mathbb{T}) \subset C(\overline{\mathbb{T}})$, such that all terms in this expression are well-defined.\\
To get some estimates on $\phi_T$, let us multiply the second equation
of \eqref{def:phi} by $\phi_T$ and integrate in space, to obtain
$$
\begin{array}{lll}
\ds \| \partial_x \phi_T\|^2_{L^2(\mathbb{T})} &=& \ds
                                                   \int_\mathbb{T}
                                                   n_i\, \phi_T\,
                                                   \dD x  \,-\, c(t)
                                                   \,\int_{\TT}
                                                   e^{\phi_T /
                                                   T}\, \phi_T\, \dD x
  \\[1.1em]
&\le & \ds{{\mathfrak N}} \, \|\phi_T\|_{L^\infty(\mathbb{T})}  \,+\,
       {\mathfrak N}\,\|\phi_T\|_{L^\infty(\mathbb{T})}
  \\[1.1em]
&\le&\ds  C\,\|\phi_T\|_{H^1(\mathbb{T})}\,,
\end{array}
$$
such that via Poincar\'e's inequality one gets that $\phi_T$ is bounded in $H^1(\mathbb{T})$ independently on $T$.

\paragraph{\bf Step 2 : Study of the map $T \mapsto {\mathcal
    E}(T)$ for fixed $t \in \RR^+$.}

To continue the proof, it will be simpler to introduce in this step the auxiliary unknown $\psi \in H^1(\mathbb{T})$, solution of the problem
\be
\label{NL_el_bis}
  \left\{
  \begin{array}{l}
\ds -\partial_{xx} \psi \,+\, {\mathfrak N}\,  e^{\psi / T} \,=\,
    n_i\,,
    \\[
    1.1em]
\ds  {{\mathfrak N} \over 2} \,T \,+\, {1 \over 2}\, \int_\mathbb{T}
    \left| \partial_x \psi \right|^2 \dD x \,+\,
    \int_\mathbb{T} w_i  \,\dD x\,=\, {\mathfrak E}\,,
\end{array}
\right.
  \ee
associated with periodic boundary conditions and the different constraint
\be
\label{cstr_bis}
\int_{\mathbb T} e^{\psi / T} \, \dD x \,=\, 1\,,
\ee
and to remark that both problems \eqref{def:ct}-\eqref{cstr} and
\eqref{NL_el_bis}-\eqref{cstr_bis} are completely equivalent.  More
precisely, we have the following relation  $\psi=\phi\,+\,K_T$ with
$$
K_T \,:=\, - T\, \ln \left(\int_{\mathbb T} e^{\phi / T}
  \, \dD x \right)\,,
$$
respectively $\phi=\psi\,+\,C_T$ with
$$
C_T\,:=\,-\int_{\mathbb T} \psi \, \dD x\,,
$$ permitting to pass from one problem to the other.
\\

Let us observe that the introduction of the auxiliary unknown $\psi_T$ does not change the map ${\mathcal E}(T)$ which writes now
\be
\label{appli_E}
{\mathcal E}:
\begin{array}{ll}
  \RR^+ & \mapsto \RR^+
          \\
  T & \ds\rightarrow  {{\mathfrak N} \over 2}\, T \,+\, {1 \over 2}\,
      \int_\mathbb{T}  \left| \partial_x \psi_T \right|^2 \dD
      x
      + \int_\mathbb{T} w_i \,\dD x\,.
\end{array}
      \ee
One can show that this application is strictly increasing and continuous in $T$, and that 
$$
{\mathcal E}(T) \rightarrow_{T \rightarrow \infty} \infty\,, \quad
{\mathcal E}(T)\rightarrow_{T \rightarrow 0}  \int_{\mathbb T} w_i \dD
x \,\le {\mathfrak E}\,.
$$
The $T \rightarrow \infty$ limit is obvious  from \eqref{appli_E}. The
continuity of ${\mathcal E}$ is based on the continuity of the map
$T\mapsto \psi_T \in H^1(\mathbb{T})$. To show this, we fix $T>0$ and hence $ \psi_T \in H^1(\mathbb{T})$ and consider the linear problem
\be
\label{theta}
-\partial_{xx} \theta \,+\,  {{\mathfrak N} \over T} \,  e^{\psi_T
  / T}\, \theta \,=\,  {{\mathfrak N} \over T^2} \,  e^{\psi_T / T} \psi_T\,,
\ee
associated with periodic boundary conditions, problem which admits a unique solution $\theta \in  H^2(\mathbb{T})$. Now, one observes that differentiating (in the distributional sense) the first equation in \eqref{NL_el_bis} with respect to $T$ yields
\be
\label{pb_psi_T}
-\partial_{xx} \left( \partial_T \psi_T \right) \,=\,  {\mathfrak N}\,  e^{\psi_T \over T}\, \left( {\psi_T \over T^2} -{ \partial_T \psi_T / T} \right)\,,
\ee
which is nothing but problem \eqref{theta}. By uniqueness one has then the existence of $ \partial_T \psi_T  \in H^2(\mathbb{T})$.\\
The continuity and monotonicity of ${\mathcal E}$ is shown by computing its derivative with respect to $T$, namely
$$
{\dD \over \dD T} {\mathcal E}(T) \,=\, {{\mathfrak N} \over 2} \,+\,
\int_\mathbb{T}   \partial_x \psi_T \,  \, \partial_{xT} \psi_T \, \dD
x \,=\, {{\mathfrak N} \over 2} - \int_\mathbb{T}    \psi_T \,  \,
\partial_{xx} \left( \partial_T \psi_T \right) \, \dD x \,.
$$
Inserting in this last formula $\psi_T$ obtained from \eqref{pb_psi_T}, yields
$$
{\dD \over \dD T} {\mathcal E}(T)\,=\, {{\mathfrak N} \over 2} \,+\,
{T^2 \over{\mathfrak N} } \,\int_{\mathbb T} e^{-\psi_T / T} |
\partial_{xx}\left( \partial_T \psi_T \right) |^2 \, \dD x \,+\, T
\,\int_{\mathbb T}| \partial_{xT} \psi_T  |^2 \, \dD x \ge 0\,.
$$
Finally, to show the limit in $T=0$ we shall investigate in more details the dependence of $\psi_T$ on $T$. For this, multiplying the equation  \eqref{NL_el_bis}  with $\psi$ and integrating in space, yields
$$
\| \partial_x \psi \|^2_{L^2(\mathbb{T})} \,+\,  {{\mathfrak N}}\,
\int_\mathbb{T} e^{\psi  / T}\, \psi\, \dD x \,=\, \int_\mathbb{T}
n_i\, \psi\, \dD x\,,
$$
thus
$$
\begin{array}{lll}
\ds \| \partial_x \psi\|^2_{L^2(\mathbb{T})} &=& \ds  \int_\mathbb{T}
                                                 n_i\, \psi\, \dD x -
                                                 {{\mathfrak N}}\,T\,
                                                 \int_{\psi \ge 0}
                                                 e^{\psi / T}\,
                                                 {\psi\over T}\, \dD x
                                                 - {{\mathfrak N}}\,T \,
                                                 \int_{\psi<0} e^{\psi
                                                 / T}\, {\psi
                                                 \over T}\, \dD x \\[1.1em]
&\le & \ds{{\mathfrak N}} \, \|\psi\|_{L^\infty(\mathbb{T})} - { {{\mathfrak N}}\over T} \, \| \psi\|^2_{L^2(\mathbb{T})} \,+\, {{\mathfrak N}}\, T\, e^{-1}\, |\mathbb{T}|\,,%\\[3mm]
%&\le& \ds  { {{\mathfrak N}}\, T \over 4 |\mathbb{T}|} +  { {{\mathfrak N}}\, T\,  |\mathbb{T}| \over e} \le \alpha \, T\,,
\end{array}
$$
where we used the fact that $e^x \ge x$ for positive $x$, whereas
$-x\, e^x \le e^{-1}$ for negative $x$. This implies $\| \partial_x
\psi \|^2_{L^2(\mathbb{T})} \rightarrow_{T \rightarrow 0} 0$, as $\|\psi\|_{H^1(\mathbb{T})} \le C$ with a constant $C>0$ independent on $T$.

Therefore, by applying Rolle's theorem, there exists a  unique $T_\star>0$ such that ${\mathcal E}(T_\star)= {\mathfrak E}$ and a unique associated $\phi_{T_\star} \in {\mathcal H}$, solution to
\eqref{def:phi} with \eqref{def:ct} and \eqref{cstr}.

This concludes the proof for fixed $t \in \RR^+$ and we shall denote
$T(t):=T_\star$. We remark additionally that $\phi_{T_\star} \in
W^{2,1}(\mathbb{T})$ and for more regular data, namely for
$n_i(t) \in L^2(\mathbb{T})$, one has even $\phi_{T_\star}(t) \in
{\mathcal H} \cap H^2(\mathbb{T})$.

Finally, the regularity in time comes now from the fact that $t$ is only a
parameter in the second equation of \eqref{def:phi}. Indeed, from the
first step,  we know that $\phi_T$ is bounded in $H^1(\mathbb{T})$ independently on $T$, and the energy conservation in \eqref{def:Te} yields the boundedness of $T(t)$, concluding the proof.
\end{proof}

%%%%%%%%%%%%%%%%%%%%%%%%%%%%%%%%%%%%%%%%%%%%%%
\section{Numerical scheme}
\label{sec:4}
\setcounter{equation}{0}
\setcounter{figure}{0}
\setcounter{table}{0}

%%%%%%%%%%%%%%%%%%%%%%%%%%%%%%%%%%%%%%%%%%%%%%% 
This section is devoted to the first steps towards the construction of a numerical scheme for
the Vlasov-Poisson-Fokker-Planck system
\eqref{eq:kin}-\eqref{constr-phi}  based on a direct discretization in
the $(x,v)$ phase-space (Eulerian approach). In the velocity space, we
shall make use of a complete, orthonormal Hermite basis to approach
the distribution functions \cite{Filbet2020,Filbet2022}. For the space
discretization a discrete Galerkin method is applied
for both transport  and  Poisson equations \cite{Ayuso2012,Filbet2020,Filbet2022}. The
use of a Hermite spectral method to discretize the velocity variable
is motivated by the fact that this method, if well scaled, is able to
reduce drastically the computational costs in situations where a
kinetic-fluid transition is investigated, thus in particular when
dealing with our adiabatic limit $\eps \rightarrow 0$. {As
  mentioned in the introduction the here presented scheme is only a
  first step towards a fully performant numerical method we shall
  present in a forthcoming work \cite{CF_bis}. In this section, we focus solely on the
  introduction of well-designed Hermite basis functions in the
  velocity space, permitting to gain considerable time in the electron
  dynamic resolution when $\eps \ll 1$, and this due to the
  possibility to reduce dynamically the number of Hermite-modes taken
  into account.} \cla{The time stiffness is a second problem, which needs
special attention and will be dealt with in a second work \cite{CF_bis}.}

%%%%%%%%%%%%%%%%%%%%%%%%%%%%%%%%%%%%%%%%%%%%%%
\subsection{Hermite expansion for the electron system}
\label{sec:4.1}
%%%%%%%%%%%%%%%%%%%%%%%%%%%%%%%%%%%%%%%%%%%%%%%
We shall start by supposing in this subsection that the ion dynamics is known, with smooth macroscopic quantities $(n_i,u_i,T_i)$, and shall present a numerical scheme solely for the resolution of the electron  Vlasov-Poisson-Fokker-Planck system
\be
\label{Ele_syst}
\left\{
 \begin{array}{l}
  \displaystyle \eps\,\partial_t f \,+\, v\, \partial_x f \,-\, E\, \partial_v f =  \cQ(f)\,,
   \\
   [1.1em]
 \ds - \partial_{xx} \phi = n_i - n, \qquad E = - \partial_x \phi\,,
    \end{array}
  \right.
\ee
with $\cQ(f)$ the mixed, non-linear Fokker-Planck operator
$$
\cQ(f) \,=\, \nu_{ee}\, \partial_v\left[ (v- u)\, f+ T\, \partial_v f \right] + \nu_{ei}\,\partial_{{v}} \left[ ({v} -  {u_{ei}}) {f} + T_{ei}\, \partial_{{v}} {f} \right],
$$
where we recall that the mixed velocities and temperatures are defined
in \eqref{defuieAdim}-\eqref{defThetaeiieAdim}. The key of the Hermite
spectral method is to construct suitable basis functions in the
velocity variable in order to cope with {the electron asymptotic limit
  $\eps \rightarrow 0$, in such a way that the limit distribution
  function is represented by only one Hermite function $\psi_0$, reducing naturally the complexity of the
  kinetic equation to the resolution of only one macroscopic equation, namely the limit model.}\\

Before performing the discretization, we recall the standard (probabilistic) Hermite polynomials $\{ J_k \}_{k \in \NN}$, which form an orthonormal basis in $L^2( \cM\, \dD v)$, where
$$
\cM(v):= {1 \over \sqrt{2 \pi}} e^{-v^2/2}\,,
$$
is the classical Maxwellian distribution function in the velocity variable. These polynomials are defined recursively as $J_0 \equiv 1$, $J_1 \equiv v$, then for any $k\geq 1$ by
$$
\sqrt{k+1}\, J_{k+1}={v}\, J_k(v) - \sqrt{k}\, J_{k-1}\,,
$$
and satisfy
$$
J_k'(v)\,=\,\sqrt{k}\, J_{k-1}(v)\,, \qquad \int_{\RR} J_k(v)\, J_l(v) \, \cM\, \dD v \,=\, \delta_{kl}\,, \quad \forall\, k,\,l \in \NN\,.
$$
For the study of the here considered Fokker-Planck
collision-operators, we have to adapt these standard Hermite
polynomials and rescale them adequately as
$$
%\be
%\label{Hpsifct}
\psi_k(v)\,:=\,{1 \over v_{th}}\, J_k\left(v\over v_{th}\right)\, \cM\left(v\over v_{th}\right)\,,
% \ee
$$
where $v_{th}$ is a scaling function to be suitably defined in the sequel, such that these Hermite basis functions are well adapted for the investigation of the asymptotic limit $\eps \rightarrow 0$.\\

To summarize, we shall expand the electron distribution function $f(t,x,v)$ as follows
\be \label{ex_H}
f(t,x,v):= \sum_{k=0}^{\infty} \alpha_k(t,x)\, \psi_k(t,v)\,,
\ee
where $\{\psi_k(t,\cdot)\}_{k \in \NN}$ forms a complete, orthonormal basis of $L^2(\cM_{v_{th}}^{-1}\, \dD v)$. {The crucial feature  of these basis functions is that} the chosen weight is given by the ``limiting'' electron Maxwellian
\be
\label{M_lim}
\cM_{v_{th}(t)}(v)\,:=\,{ 1 \over \sqrt{2 \pi}\, v_{th}(t)}  \exp\left(-{v^2 \over 2 v_{th}^2(t)}\right)\,, \qquad v_{th}(t) := \sqrt{T(t)}\,,
\ee
with $T(t)$ the electron temperature given by the asymptotic limit model \eqref{def:Te}-\eqref{cstr}.\\
These Hermite basis functions are defined recursively as $\psi_0(t,v) \equiv \cM_{v_{th}(t)}(v)$,  $\psi_1(t,v) \equiv v \cM_{v_{th}(t)}/v_{th}(t)$ and for $k\geq 1$ by
\be \label{Heri}
\sqrt{k+1}\, \psi_{k+1}(t,v)\,=\,{v \over v_{th}}\, \psi_k(t,v) \,-\, \sqrt{k}\, \psi_{k-1}(t,v)\,, 
\ee
and satisfy
$$
v_{th}\, \partial_v \psi_k(t,v)\,=\,-\sqrt{k+1}\, \psi_{k+1}(t,v)\,, \qquad \int_{\RR} \psi_k(t,v)\, \psi_l(t,v) \, \cM^{-1}_{v_{th}}\, \dD v \,=\, \delta_{kl}\,, \quad \forall \,k,\,l\, \in\, \NN\,.
$$
One should be aware that different scalings of the basis functions lead to different expansion series. Based on a priori knowledge on the behaviour of our solution, the scaling is chosen in such a manner to enable the convergence towards the desired equation, in our case towards the limit model as $\eps \rightarrow 0$. { Hence it is important to underline here that for the definition of the Hermite basis functions $\{\psi_k(t,\cdot)\}_{k \in \NN}$ one should first solve the limit model \eqref{def:Te}-\eqref{cstr} in order to get the scaling factor $v_{th}(t)$.}

The coefficients $\{\alpha_k(t,x)\}_{k \in \NN}$ in \eqref{ex_H} are still to be determined, and are given in the following proposition.
\begin{prop}
Let $(f,\phi)$ be the electron distribution function resp. the potential, solution of the system \eqref{Ele_syst}, and let us consider the decomposition \eqref{ex_H} in the Hermite basis functions given in \eqref{M_lim}-\eqref{Heri}. Then the Hermite coefficients $\{\alpha_k\}_{k \in \NN}$ are solutions to the following coupled, nonlinear, infinite PDE-system
\be \label{PDE_syst}
\left\{
\begin{array}{l}
  \ds \eps\, \left(\partial_t \alpha_k\,+\,{v_{th}'\over v_{th}} \left[k \,\alpha_k \;+\, \sqrt{(k-1)k}\, \alpha_{k-2} \right]\right) \,+\, v_{th}\,\partial_x\left(\sqrt{k}\,\alpha_{k-1} +\sqrt{k+1}\, \alpha_{k+1}\right)
  \\[1.1em]
  \ds  +\,\sqrt{k} \, {E \over v_{th}}\,\alpha_{k-1} \,+\, \nu_{ee}\,  \left( k \,\alpha_k - \sqrt{k} {u \over v_{th}} \alpha_{k-1} + \left[1 - {T \over v_{th}^2} \right]\sqrt{(k-1)\,k}\, \alpha_{k-2}\right)
  \\[1.1em]
  \ds\,+\, \nu_{ei}\, \left( k \,\alpha_k - \sqrt{k} {u_{ei} \over v_{th}} \alpha_{k-1} + \left[1 - {T_{ei} \over v_{th}^2} \right]\sqrt{(k-1)\,k}\, \alpha_{k-2} \right) \;=\,0\,,
\end{array}
\right.
\ee
where we set $\alpha_{l} \equiv 0$ for $l<0$, whereas the electron
momentum $nu$ and temperature $T$ are linked with the first Hermite expansion coefficients via the formulae
\be
\label{moments}
 n\, u   \,=\, v_{th}\,\alpha_1\,, \quad T= v_{th}^2 \left[1+ \sqrt{2}\, {\alpha_2 \over \alpha_0} - \left({\alpha_1\over \alpha_0}\right)^2 \right]\,.
\ee
This system is coupled to Poisson's equation
\be \label{Poi_syst}
- \partial_{xx} \phi \,=\, n_i- \alpha_0\,, \qquad E = - \partial_x \phi\,.
\ee
\end{prop}
\begin{proof}
Multiplying \eqref{Ele_syst} by $\alpha_l$ and taking the scalar-product in $L^2(\cM_{v_{th}}^{-1}\, dv)$, yields immediately the system \eqref{PDE_syst}. Observing then that $\langle \psi_k \rangle =0$ for all $k \neq 0$, {fact which is obtained recursively from \eqref{Heri}}, one has with \eqref{ex_H} that 
$$
n(t,x)= \int_{\RR} f(t,x,v)\, dv=\alpha_0(t,x)\,,
$$
leading to the form \eqref{Poi_syst} of Poisson's equation.
\\
Finally, the fact that $\{\psi_k(t,\cdot)\}_{k \in \NN}$ forms an orthonormal basis in $L^2(\cM_{v_{th}}^{-1}\, dv)$ means that the hierarchy \eqref{PDE_syst}-\eqref{Poi_syst} is equivalent to \eqref{Ele_syst} and admits thus a unique solution $\{ \alpha_k (t,x) \}_{k \in \NN} $.
\end{proof}
There are several advantages when using a Hermite spectral method for the discretization of the velocity variable. Firstly the functions $\{\psi_k\}_{k \in \NN}$ form a complete, orthonormal basis  of $L^2(\cM^{-1}_{v_{th}}\, \dD v)$ with respect to the Gaussian weights, such that these basis functions seem to be optimal to approach Maxwellian-like distribution functions in the velocity variable. Secondly, the lower-order terms in the expansion \eqref{ex_H} are related to the low order moments of the distribution function, meaning to the macroscopic quantities like the density, the momentum and the energy, quantities, which are usually of interest. The kinetic features of the problem are retained by considering more modes in the Hermite expansion \eqref{ex_H}.
Thus, such a Hermite spectral method permits somehow to make the link between the kinetic and the fluid descriptions, and is particularly well suited for our asymptotic study $\eps \rightarrow 0$.\\

To be more precise, one can observe that $\langle v\, \psi_k \rangle
=0$ for all $k\ge 2$ and $\langle v^2\, \psi_k \rangle =0$ for all
$k\ge 3$, {obtained again recursively from \eqref{Heri}, such that the electron
momentum $nu$, energy $w$ and temperature $T$ are linked with the
first Hermite expansion coefficients via the formulae \eqref{moments}}.
Substituting these expressions in \eqref{PDE_syst}, the first three
equations for $(\alpha_0,\alpha_1,\alpha_2)$ are given by \cite{Filbet2020,Filbet2022}.
$$
\left\{
\begin{array}{l}
\ds\eps\, \partial_t \alpha_0 \,+\, v_{th} \,\partial_x \alpha_1\,=\,0\,,
\\[1.1em]
\ds\eps \,\left(\partial_t \alpha_1 \,+\, {v_{th}' \over v_{th}}\, \alpha_1\right) \,+\,  v_{th}\,\partial_x \left(\alpha_{0} + \sqrt{2}\, \alpha_{2}\right) \,+\, {E \over v_{th}}\,\alpha_{0} \,=\, -\frac{\nu_{ei}}{2}\, \left( \alpha_1 - \eps  { u_i \over v_{th}}\, \alpha_0 \right)\,,
\\[1.1em]

\ds\eps\, \left(\partial_t \alpha_2\,+\,\sqrt{2}\, {v_{th}'\over v_{th}} \left[\alpha_{0}+\sqrt{2} \alpha_2  \right]\right) \,+\, v_{th}\,\partial_x\left(\sqrt{2}\,\alpha_{1} + \sqrt{3}\, \alpha_{3}\right) \,+\,\sqrt{2} \, {E \over v_{th}}\,\alpha_{1}\,=\, \, {\tilde{S}}_{ei}\,,
\end{array}\right.
$$
with
\begin{equation*}
%  \label{def:Sei}
{\tilde{S}}_{ei} \,=\, - \nu_{ei}\,\alpha_0\,{\sqrt{2} \over v_{th}^2}\,\left[ T_e- T_{ei} +\, u_e\,\frac{u_e- \eps\,u_i}{2}\right]\,,
\end{equation*}
which correspond exactly to the first three moment equations \eqref{fluid}. Taking into account for both species, we recover the conservations of mass,  momentum and total energy 
$$
\left\{
\begin{array}{l}
\ds\frac{\dD }{\dD t}\int_{\TT}\alpha_0\,\dD x  \,=\,0\,,
\\[1.1em]
\ds\frac{\dD}{\dD t}\int_{\TT} \left[\eps\, v_{th}\,\alpha_1 \,+\, n_i\,u_i\right] \,\dD x  \,=\,0\,,
\\[1.1em]
\ds\frac{\dD}{\dD t} \int_{\TT} \left[{v_{th}^2 \over 2} \left(\alpha_0 + \sqrt{2}\,\alpha_2\right) \,+\, \frac{1}{2}\,|\partial_x \phi|^2 \,+\, w_i\right]\,\dD x\,=\,0\,.
\end{array}\right.
$$
These constraints are automatically fulfilled for $\eps >0$, however in the limit they have to be imposed in order to get uniqueness of the limit model.
Taking formally the limit $\eps \rightarrow 0$ in \eqref{PDE_syst} we get the following Proposition.
\begin{prop}
  \label{prop:5}
In the limit $\eps \rightarrow 0$ the Hermite coefficients $\{ \alpha_k^\eps (t,x) \}_{k \in \NN} $, solutions to the system \eqref{PDE_syst}-\eqref{Poi_syst}, tend towards some coefficients $\{ \alpha_k\}_{k \in \NN}$, satisfying the following limit PDE-system 
 \be \label{ttt}
\left\{
\begin{array}{l}
  \ds  v_{th}\,\partial_x\left(\sqrt{k}\,\alpha_{k-1} \,+\,\sqrt{k+1}\,\alpha_{k+1}\right)\,+\,\sqrt{k} \, {E \over v_{th}}\,\alpha_{k-1}
  \\[1.1em]
\ds \,+\, \nu_{ee}\, \left( k \,\alpha_k\,-\, \sqrt{k}\,  {\alpha_1 \over \alpha_0}\,  \alpha_{k-1} \,-\; \left[\sqrt{2}\,{\alpha_2 \over \alpha_0} \,-\, \left({\alpha_1\over \alpha_0}\right)^2\right]\,\sqrt{(k-1)\,k}\, \alpha_{k-2} \right)
\\[1.1em]
\ds \,+\,\nu_{ei}\, \left( k \,\alpha_k \,-\, \sqrt{k}\, {\alpha_1 \over 2\, \alpha_0}\, \alpha_{k-1} \,-\, \left[\sqrt{2}\, {\alpha_2 \over \alpha_0} - {1\over 2}\, \left({\alpha_1\over \alpha_0}  \right)^2  \right]\sqrt{(k-1)\,k}\, \alpha_{k-2} \right)\,=\,0\,,
\\[1.1em]
\ds \int_{\TT}\alpha_0\,\dD x \,=\,{\mathfrak N}\,, \qquad  {1\over 2} \int_{\TT}\left[v_{th}^2\left(\alpha_0 + \sqrt{2}\, \alpha_2\right) + |\partial_x \phi|^2 + 2\,w_i\right]\,\dD x \,=\, {\mathfrak E}\,,
\end{array}
\right.
 \ee
coupled to Poisson's equation
\be
\label{ttt2}
- \partial_{xx} \phi \,=\, n_i\,-\, \alpha_0\,, \qquad E \,=\, - \partial_x \phi\,.
\ee
This system admits a unique solution $\{ \alpha_k (t,x) \}_{k \in
  \NN}$, given by $\alpha_k \equiv 0$ for all $k \neq 0$ and for the zero$^{th}$ order coefficient by the equation
$$
v_{th}\,\partial_x \alpha_{0} +\, {E \over v_{th}}\,\alpha_{0}=0\,,
$$
with $\phi$ given by Poisson's equation and $v_{th}(t)=\sqrt{T_e}(t)$ by the energy equation in \eqref{def:Te}. 
\end{prop}
\begin{proof}
 {We simply observe that \eqref{ttt}-\eqref{ttt2}  is nothing else than
 the electron Maxwell-Boltzmann relation given by the well-posed limit system \eqref{def:MB}-\eqref{eq:lim2}, equivalently rewritten in Remark \ref{remark_L} under the form \eqref{remark_L}.}
  \end{proof}
%%%%%%%%%%%%%%%%%%%%
\subsection{Space/time discretization}
\label{sec:4.2}
\newcommand{\xR}{{x_{i+\frac{1}{2}}}}
\newcommand{\xL}{{x_{i-\frac{1}{2}}}}
\newcommand{\iL}{{i-\frac{1}{2}}}
\newcommand{\iR}{{i+\frac{1}{2}}}
\newcommand{\f}{\frac}
\newcommand{\testR}{{\varphi}}   % test function
\newcommand{\testP}{{\eta}}   % test function
\newcommand{\testE}{{\zeta}}
\newcommand{\px}{\partial x}
In the spirit of \cite{Ayuso2012,Filbet2020,Filbet2022}, we consider now a discontinuous Galerkin approximation for the space discretization of  the Vlasov-Fokker-Planck equation, written via the Hermite basis functions under the form \eqref{PDE_syst}.

We first introduce some notation and start with
$\{\xR\}_{i=0}^{i=N_x}$, a partition of ${\mathbb T}=(0,L)$, with $x_{\frac12}=0$,
$x_{N_x+\frac12}=L$. Each element is denoted by $I_i=[\xL,
  \xR]$ with length $h_i$ and
$$
h\,=\,\max_i h_i\,.
$$
For any $l\in\mathbb{N}$ we introduce the finite dimensional discrete, piecewise polynomial space 
\begin{equation}
  \label{defVhk}
V_h^l=\left\{u\in L^2(0,L),\quad u|_{I_i}\in P_l(I_i), \quad\forall i =0, \ldots, N_x\right\}\,,
%\label{eq:DiscreteSpace:1mesh}
\end{equation}
where the local space $P_l(I)$ denotes the set of polynomials of degree at
most $l$ on the interval $I$. In the here presented simulations we used second order polynomials, {\it i.e.} $l=2$.
We further define for any $i\in\{0,\ldots,N_x\}$,  the jump $[u]_\iR$
and the average $\{u\}_\iR$ of $u$ at $x_\iR$  as
$$
[u]_\iR\,:=\,{u(x_\iR^+)\,-\,u(x_\iR^-)} \quad{\rm and}\quad
\{u\}_\iR\,:=\,\frac12\,\left(u(x_\iR^+)\,+\,u(x_\iR^-)\right)\,, 
$$
where $u(x^\pm):=\lim_{\Delta x\rightarrow 0^\pm} u(x+\Delta x)$.  We
also set

\[  u_\iR=u(x_\iR)\,, \qquad u^\pm_\iR=u(x^\pm_\iR)\,. \]

The approximate solution of \eqref{Ele_syst}, obtained using Hermite polynomials in the velocity variable and a discontinuous Galerkin discretization in the space variable, is reconstructed as
\be
\label{dgfseries}
f_{h}(t,x,v)=\sum_{k=0}^{N_H-1}\alpha_{k,h}(t,x)\psi_k(t,v)\,,
\ee
where  $\{\psi_k\}_k$ are the basis functions defined by \eqref{Heri} and the set $\{\alpha_{k,h}\}_k$ is determined by the discontinuous Galerkin method, employed for solving  \eqref{PDE_syst} and presented in the following. The truncation index $N_H \in \NN$ \cla{ will be adapted, considering the vicinity to the limit model, as explained in Section \ref{sec:4.3}. This reduces the computational costs.}\\

On one hand, we look for an approximation $\alpha_{k,h}(t,\cdot) \in V_h^l$, such that for any $\varphi_k \in V_h^l$, we have
\be
  \eps\,\frac{\dD}{\dD t}\int_{I_j}\alpha_{k,h}\,\varphi_k\,\dD x \,=\, b_k^j(E_h,\alpha_h,\varphi_k) \,+\, a_k^j(g_k,\varphi_k),\quad 0\leq k \leq N_H-1,
  \label{dgcn}
\ee
where $b_k^j$ is an approximation of the source terms of \eqref{PDE_syst} 
\be
\label{bnh}
\left\{
  \begin{array}{ll}
    \ds b^j_{k}(E_h,\alpha_h,\varphi_k) &\ds =\, -\int_{I_j}\left[\eps\,\cI_k k(\alpha_h)  + \frac{\sqrt{k}\,E_{h}}{v_{th}}\,\alpha_{k-1,h}\,+\, \cQ_k k(\alpha_h) \right]\varphi_k\,\dD x\,,
    \\[1.1em]
    \ds \cI_k k(\alpha_h)  &\ds =\, \frac{v_{th}'}{v_{th}}\left(k\,\alpha_{k,h}+\sqrt{(k-1)k}\,\alpha_{k-2,h}\right)\,,
    \\[1.1em]
    \ds\cQ_k k(\alpha_h)  &\ds =\, (\nu_{ee}+\nu_{ei})\,k \,\alpha_k^\eps \,-\, \sqrt{k}\, {\nu_{ee}u^\eps+\nu_{ei}u_{ei}^\eps \over v_{th}} \,\alpha_{k-1}^\eps
    \\[1.1em]
 \,&\ds\ + \,\left(\nu_{ee}\left[1 - {T^\eps \over v_{th}^2}\right]+\nu_{ei}\left[1 - {T_{ei}^\eps \over v_{th}^2} \right]\right)\,\sqrt{(k-1)\,k}\, \alpha_{k-2}^\eps\,,
\end{array}\right.
\ee
whereas $a^j_k$ represents the space derivative approximation,  defined by 
\be
\label{anh}
\left\{
  \begin{array}{l}
\ds a^j_{k}(g_{k},\varphi_k) \,=\, -\int_{I_j}g_{k}\,\testR'_k \,
    \dD
    x\,+\,\hat{g}_{k,j+\f12}\,\testR^-_{k,j+\f12}-\hat{g}_{k,j-\f12}\,\testR^+_{k,j-\f12}\,,
    \\[1.1em]
\ds g_{k} \,=\,v_{th}\,\left(\sqrt{k+1}\,\alpha_{k+1,h}\,+\,\sqrt{k}\,\alpha_{k-1,h}\right)\,.
\end{array}\right.
\ee

The numerical flux $\hat{g}_{k}$ in \eqref{anh} is given by
\be
\label{lfflx}
\hat{g}_{k}\,=\,\f12\left[g^-_{k}+g^+_{k}-\delta_k\,\left(\alpha^+_{k,h}\,-\,\alpha^-_{k,h}\right)\right]\,,
\ee
with the numerical viscosity coefficient defined for $k=0$ as $\delta_0=0$, corresponding to a centered flux, and for $1\leq k\leq N_H-1$ we consider the global Lax-Friedrichs flux with $\delta_k=\delta=\sqrt{N_H}/\alpha_{N_H,h}$. The choice of the centered flux in the case $k=0$ is made to recover the conservation of the semi-discrete total energy.\\

On the other hand, we  search for an approximation of the electric field $E_{h}$.  To this end, we need to consider the potential function $\phi_{h}(t,x)$, such that
\begin{equation}\label{approxPoisson}
\left\{
\begin{array}{l}
\ds E_{h}\,=\,-\f{\partial \phi_{h}}{\px}\,, \\[0.9em]
\ds\f{\partial E_{h}}{\px} \,=\, n_i - \alpha_{0,h}\,,
\end{array}\right.
\end{equation}
which is equivalent to the one dimensional Poisson equation 
$$
-\f{\partial^2 \phi_{h}}{\px^2} \,=\,n_{i,h} \,-\, \alpha_{0,h}\,.
$$
Let us discretize \eqref{approxPoisson} via a discontinuous Galerkin approximation. For this, we
look for a couple $(\phi_{h}(t,\cdot), E_{h}(t,\cdot)) \in V_h^l \times V_h^l $, such
that for any $\testP$ and $\testE$ belonging to  $ V_h^k$, we have 
\be
\label{dgP}
\left\{
\begin{array}{l}
\ds+\int_{I_j}\phi_{h}\,\testP'\,\dD x \, -\,
  \hat{\phi}_{h,j+\f12}\,\testP^-_{j+\f12} \,+\,\hat{\phi}_{h,j-\f12}\,\testP^+_{j-\f12}\,=\,\int_{I_j}
  E_{h}\,\testP\,\dD x\,,
  \\[1.1em]
\ds -\int_{I_j}E_{h}\,\testE'\,\dD x \,+\,
  \hat{E}_{h,j+\f12}\,\testE^-_{j+\f12}
  \,-\,\hat{E}_{h,j-\f12}\,\testE^+_{j-\f12}\,=\,\int_{I_j} \left(
  n_{i,h}  - \alpha_{0,h}\right)\,\testE\,\dD x\,,
\end{array}\right.
\ee
where the numerical fluxes $\hat{\phi}_{h}$ and $\hat{E}_{h}$ in
\eqref{dgP}  are taken as 
\be
\label{fluxps}
\left\{
\begin{array}{l}
  \hat{\phi}_{h} \,=\, \{\Phi_{h}\}\,, \\[0.9em]
  \hat{E}_{h} \,=\, \{E_{h}\}\,-\,\beta\,[\Phi_{h}]\,,
\end{array}\right.
  \ee
with $\beta$ being a positive constant, possibly proportional to $1/h$ (see \cite{cockburn} for more details).\\

Finally, the last free parameter is $v_{th}(t)$. It is chosen such
that our numerical discretization captures well the limit
$\eps\rightarrow 0$. Therefore, following Proposition \ref{prop:5}, we
choose $v_{th}(t)$ such that the energy conservation is satisfied in
the limit $\eps\rightarrow 0$, namely via the equation for $(\overline{\alpha}_0,\overline{\phi})$,
\be
\label{def:vth}
{1\over 2} \int_{\TT}\left[v_{th}^2\,\overline{\alpha}_0 \,+\, |\partial_x \overline{\phi}|^2 \,+\, 2\,w_i\right]\,\dD x \,=\, {\mathfrak E},
\ee
where $\overline{\phi}$ solves
$$
\left\{
\begin{array}{l}
\ds\overline{\alpha}_0\,=\,\frac{{\mathfrak N}}{\int_{\TT}e^{\overline{\phi}/v_{th}}\dD x}\, e^{\overline{\phi}/v_{th}}\,,
 \\[1.2em]
 \ds- \partial_{xx} \overline{\phi} = n_i- \overline{\alpha}_0.
 \end{array}\right.
$$
This choice will guarantee that in the limit $\eps \rightarrow 0$ the coefficients $\{\alpha_{k,h}^\eps\}_k$ will be consistant with the
Maxwell-Boltzmann distribution \eqref{def:MB}.\\

Concerning the time-discretization, we apply a second-order Crank-Nicolson scheme to the just introduced discontinuous Galerkin method. We denote by $\mathbf \alpha=(\alpha_0,\ldots,\alpha_{N_H-1})$ the solution to \eqref{dgcn}--\eqref{lfflx} and by $(\cdot,\cdot)$ the standard $L^2$ inner-product on the space interval $(0,L)$, namely
\[(\alpha_n,\varphi):=\int_0^L\alpha_n\,\varphi\,\dD x\,,\]
and let $\Delta t>0$ be the time step.

Furthermore let  $\mathbf{\alpha}_h^m=(\alpha_0^m,\ldots,\alpha_{N_H-1}^m)$ be the approximation of the solution $\mathbf{\alpha}$ at time $t^m=m\,\Delta t$ for $m\geq 0$,  and let us denote, for an arbitrary variable $\xi$
$$
\xi^{m+1/2} := \f12\,\left(\xi^m+\xi^{m+1}\right)\,.
$$

Assuming known $\alpha_h^{m}$, we compute now $\alpha_k^{m+1}$ for $k=0,\ldots,N_H-1$ via
\be
\label{c-n}
\f{ (\alpha_k^{m+1}-\alpha_k^{m},\varphi_k )}{\Delta t} +a_k(g_k^{m+1/2},\varphi_k) +
b_k(v_{th}^{m+1/2},\alpha^{m+1/2}_{h},E^{m+1/2},\varphi_k) = 0\,, \quad \forall \varphi_k \in V_h^l\,,
\ee
and solve the DG approximation of the Poisson equation \eqref{dgP}--\eqref{fluxps} to obtain $E^{m+1}$.

\cla{It is worth to emphasize that this time discretization is not
necessarily uniformly stable with respect to $\eps$, $h$ and
  $N_H$.  For an explicit scheme the CFL condition for the system
  \eqref{dgcn}--\eqref{fluxps} would be
  $$
\Delta t = \mathcal{O}\left(\frac{\eps\,h}{\sqrt{N_H}}\right),
$$
where $N_H$ denotes the number of Hermite coefficients taken into account.
Here an iterative solver is applied to solve this nonlinear and
ill-conditionned system. In practice, to
ensure the convergence of the iterative method, the time-step
  $\Delta t$ is still dependent on $\eps$ and $(h,\,N_H)$. An AP-scheme is further needed to cope with this
  problem and this requires a further study \cite{CF_bis}. However the
  Crank-Nicolson scheme  is well adapted to our approach where the preservation of energy plays a key role.}

%%%%%%%%%%%%%%%%%%%%
\subsection{Adaptive algorithm for Hermite coefficients}
\label{sec:4.3}
\cla{In this subsection, we provide a simple and efficient way to reduce
the computational complexity. Indeed, due to the choice of the scaling function  $v_{th}$, we expect
that when the solution of the kinetic equation approaches a
hydrodynamical regime,  the higher order Hermite coefficients will
rapidly converge to zero
and can then be neglected. Therefore we propose a simple adaptive
algorithm at each time iteration to remove small coefficients (see Algorithm
\ref{alg:1}). Briefly speaking, a Hermite coefficient can be neglected
when this coefficient and its neighbours  are small. Otherwise, it
should be considered. This adaptive
procedure allows to remove or  to add
Hermite coefficients dynamically.  As a consequence, the
computational complexity will be considerably reduced in the asymptotic regime $\eps\rightarrow 0$
even when the initial datum is not well prepared. Indeed, the initial
time layer will be described correctly using a large number of
Hermite coefficients, while when the solution approaches the equilibrium
fewer and fewer coefficients will be used.

\begin{algorithm}
%  \caption{Adaptive algorithm for the computation of the Hermite coefficients at time step $t^{m+1}$, supposing everything knwon at time step $t^{m}$.}
\For{$k= 3,\ldots$}{
  \eIf{$\|\alpha_k^m\|_{\infty}$, \,  $\|\alpha_{k+1}^m\|_{\infty}$, \,
    $\|\alpha_{k-1}^m\|_{\infty} \leq 10^{-6}$ }{ set $I_k^m = 0$
    \Comment{Contribution of $\alpha_k^m$ will be neglected\,\,} }{ set $I_k^m = 1$
    \Comment{Contribution of $\alpha_k^m$ will be considered} }
}
%\,\\
Solve the system \eqref{dgcn}--\eqref{lfflx} only for
$(\alpha_k^{m+1})_{k}$ such that $I_k^m=1$, whereas we set $\alpha_k^{m+1}=0$ for $I_k^m=0$.
\\
\, \quad 
\\
\caption{Adaptive algorithm for the computation of the Hermite coefficients at time step $t^{m+1}$, supposing everything known at time step $t^{m}$.}\label{alg:1}
\end{algorithm}
}

%%%%%%%%%%%%%%%%%%%%
\section{Numerical simulations}
\label{sec:5}
\setcounter{equation}{0}
\setcounter{figure}{0}
\setcounter{table}{0}

%%%%%%%%%%%%%%%%%%%%
{In this section we shall present numerical simulations based
on the scheme proposed above  to  investigate
in more details the adiabatic regime when $\eps \ll 1$ for the Vlasov-Poisson Fokker-Planck
system \eqref{Ele_syst}. Our aim  is to focus on  weakly collisional plasmas, where collective
effects, due to the transport part, dominate collisional effects. In
this situation, we  illustrate  what happens in this $\eps \rightarrow 0$ adiabatic asymptotic, and in particular we are interested in investigating what is the advantage of using a Hermite
spectral-method in the velocity variable, for physically relevant mass ratio $\eps^2={m_e}/{m_i}$.}

%%%%%%%%%%%%%%%%%%%
\subsection{One species case}
\label{sec:5.1}
{In this first example, we examine only the electron evolution, the ions
being considered as forming a sort of fixed background, interacting
with the electrons only via the electric field.  Since in this case,
the limit $\eps\rightarrow 0$ corresponds to the asymptotic limit
$t\rightarrow \infty$, we fix $\eps$ to one and investigate the long
time behavior of the following equation  
$$
\left\{
 \begin{array}{l}
  \displaystyle \partial_t f \,+\, v\, \partial_x f \,-\, E\, \partial_v f =  \nu_{ee}\, \partial_v\left[ (v- u)\, f+ T\, \partial_v f \right] \,,
   \\
   [1.1em]
 \ds - \partial_{xx} \phi = n_i - n, \qquad E = - \partial_x \phi\,,
    \end{array}
  \right.
$$
with $\nu_{ee}=0.01$, which corresponds to a weakly collisional
plasma.  Also the background ion density is considered as time-independent and given by
$$
n_i(x)  \,=\; 1+\kappa\,\cos(k\,x)\,, \qquad \forall x \in (0,L)\,,
$$
with  $k=2\pi/L$, $L=12$ and $\kappa=0.1$, whereas the electron
initial distribution 
function $f_0$ is given by
\be
\label{2stream}
f_0(x,v)\,=\,\frac{1}{6\sqrt{2\pi}}\,  (1+5\,v^2) \,
\exp\left(-\frac{|v-u_0(x)|^2}{2}\right)\,, \qquad \forall (x,v) \in (0,L) \times \RR\,,
\ee
with
$u_0(x) = 0.5 \,\sin(k\,x)$. Let us emphasize that the distribution
function $f$ is initially far from the thermal equilibrium since it
corresponds to two streams in the velocity variable with a nonzero
mean velocity $u_0$. Moreover, since we do not consider in this test case collisions between electrons and
ions, this equation conserves mass,
momentum and total energy, hence the assumption that
$$
\int_0^L \int_\RR v f_0^\eps(x,v) \,\dD v \,\dD x = 0\,,
$$
is mandatory to get the convergence of the distribution function, when $t\rightarrow
\infty$,  towards a stationary state given by the Maxwell-Boltzmann
distribution \eqref{def:MB}. The  scaling parameter necessary for the definition of the Hermite basis functions is chosen as
$v_{th}=\sqrt{\overline{T}}$, where  the temperature $\overline{T}$
and the potential $\overline{\phi}$ correspond to the stationary
solution of the limit model
\be
\label{def:Tbar}
\left\{
\begin{array}{l}
\ds\frac{{\mathfrak N}}{2}\, \overline{T} \,+\,\frac{1}{2}\,\int_{\TT}
|\partial_x \overline{\phi}(x)|^2 \,\dD x\, =\, {\mathfrak E}\,,
\\
\ds- \partial_{xx} \overline{\phi} \,+\, c\,\exp\left( \frac{\overline{\phi}}{\overline{T}}\right)\,=\, n_i\,,
  \end{array}\right.
\ee
with $c$ uniquely determined by the conservation of the particle number. This
latter system has to be solved initially.
\\
We performed several numerical simulations using  the discontinuous
Galerkin/Hermite method  and refining  the mesh and
the time step $\Delta t$ as in \cite{Filbet2020, Filbet2022}, but for
the sake of clarity we only report  numerical simulations with
$N_x\times N_H=32\times 64$ and  $\Delta t =1/500$ for which the
numerical results are similar with those obtained with refined
meshes. Since initially the solution is far from equilibrium and
collective effects dominate, the adiabatic asymptotics is not valid in a
transient regime, hence a large number of modes $N_H$ is needed to
describe kinetic effects.
\\
On one hand, we show  on Figure \ref{fig:11} $(a)$ the time
evolution of the deviations with respect to the initial condition of
the discrete mass, momentum and total energy and  observe
that  the errors on these quantities are of order $5\,\times 10^{-8}$  (our space/time discretization does not ensure
exact conservations) which is acceptable. We also present in Figure
\ref{fig:11} $(b)$ the convergence in time of the potential $\Phi$ towards its
equilibrium $\overline{\Phi}$. The amplitude of the potential $\Phi$ first
 oscillates strongly, and then, for times $t\geq 20$, it is damped and
converges to the stationary state $\overline{\Phi}$ given by \eqref{def:Tbar}. 
\\
On the other hand, we  present in Figure \ref{fig:12} $(a)$ the time
evolution of the $L^2$-norm of the mean electron velocity as well as
of the temperature-deviation with respect to the temperature equilibrium
$v_{th}^2$ in log scale. As for the potential, we observe a transient regime, where
collective effects dominate, then both quantities converge to zero
exponentially fast.  In Figure
\ref{fig:11}(b) we  present then the time evolution of the  $L^2$-norm of the Hermite
coefficients $(\alpha_k)_{1\leq k \leq 6}$.  All these coefficients decrease almost exponentially
fast to zero, saturating then around $10^{-12}$. This illustrates
the convergence of our electron distribution function (when
$t\rightarrow \infty$) towards the Maxwell-Boltzmann distribution for which all Hermite
coefficients are zero, except the main one $\alpha_0$. This is made possible
by the appropriate choice of the scaling parameter $v_{th} = \sqrt{\overline{T}}$ according
to \eqref{def:Tbar}. These results show that in practice, higher-order
Hermite coefficients can be neglected  when
$t$ becomes large  and  thus the truncated Hermite hierarchy reduces to a consistant
approximation of the limit system \eqref{def:MB}-\eqref{cstr}. It
illustrates the efficiency of our algorithm passing automatically from the numerical
resolution (and complexity) of the kinetic equation when the solution
is far from a Maxwell-Boltzmann distribution to the adiabatic limt,
where only one mode is used for the density.}

\begin{figure}
  \centering
  \begin{tabular}{cc}
	\includegraphics[width=3.2in,clip]{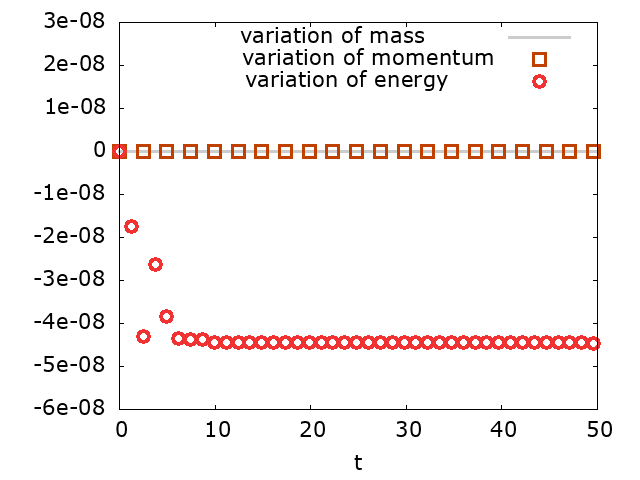}&
	\includegraphics[width=3.2in,clip]{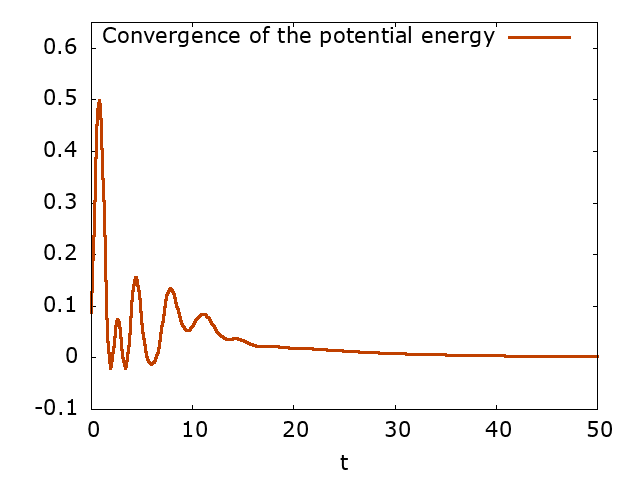}
        \\
        (a) & (b)
        \end{tabular}
	\caption{{\bf One species case :} $(a)$ deviation of mass,
          momentum and energy as compared to the initial condition; $(b)$  time
          evolution of the potential  deviations from the asymptotic
          value $\overline{\Phi}$
          for  $N_x\times N_H = 32 \times 64$.}
	\label{fig:11}
\end{figure}

\begin{figure}
  \centering
  \begin{tabular}{cc}
	\includegraphics[width=3.2in,clip]{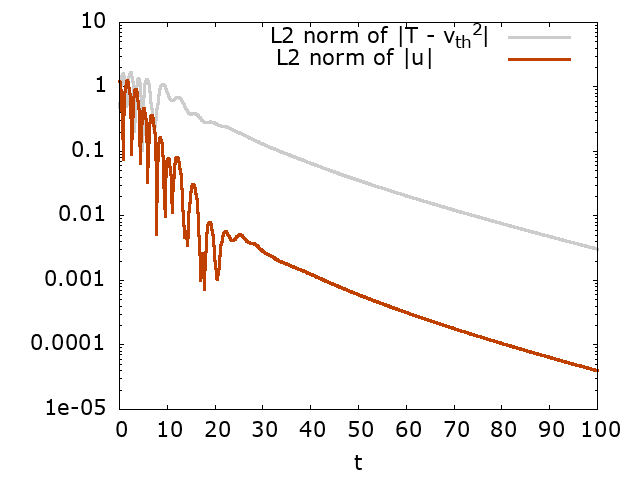} &
	\includegraphics[width=3.2in,clip]{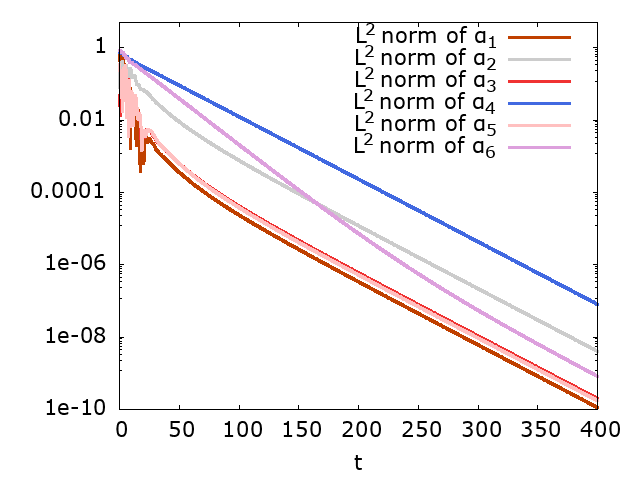}\\
        (a)&(b)
        \end{tabular}
	\caption{{\bf One species case:} $(a)$ time evolution (in log-scale)  of $\|u\|_{L^2}$ and
          $\|T-v_{th}^2\|_{L^2}$; $(b)$ time evolution of the Hermite-coefficients $(\alpha_k)_{1\leq k \leq 6}$ in log-scale, for $N_x\times N_H = 32 \times 64$.}
	\label{fig:12}
\end{figure}

Finally we show on Figure \ref{fig:14} some snapshots of the electron distribution function corresponding to  the transient regime.  Indeed,
when $t\leq 12.5$, collective effects dominate and the distribution
function starts to develop some filaments in phase space then for larger times the electric field is
damped (due to collisional effects) and $f$ converges to the space non homogeneous
Maxwell-Boltzmann equilibrium when $t\rightarrow \infty$.  \\

\begin{figure}
  \centering
  \begin{tabular}{cc}
	\includegraphics[width=3.2in,clip]{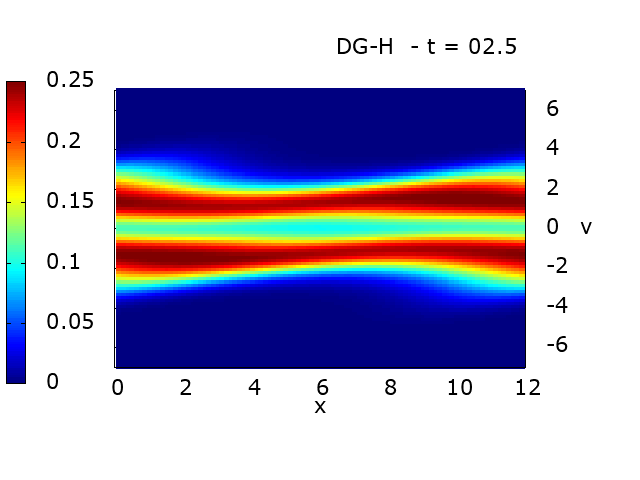} &
	\includegraphics[width=3.2in,clip]{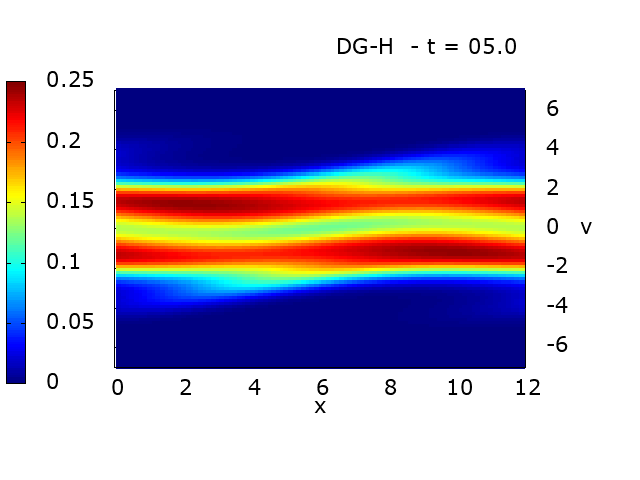}
        \\
        \includegraphics[width=3.2in,clip]{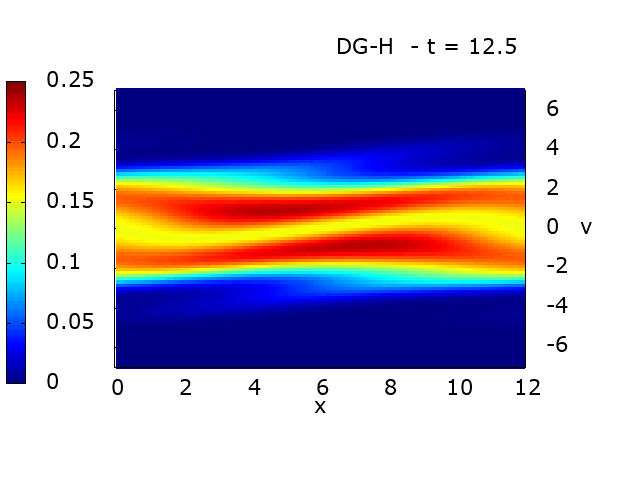} &
	\includegraphics[width=3.2in,clip]{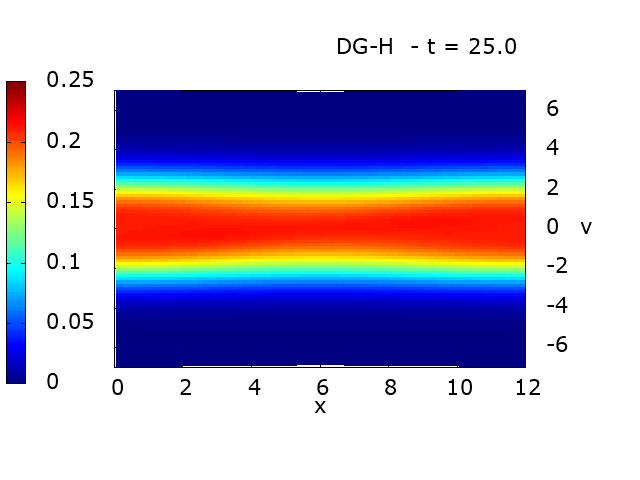}
\end{tabular}
	\caption{{\bf One species case:}  Surface plot of
          the distribution function $f$ at several times $t=2.5$, $5$, $12.5$ and $25$, with  $N_x\times N_H = 32 \times 64$.}
	\label{fig:14}
      \end{figure}
      
%%%%%%%%%%%%%%%%%%%%%%%%%%%%%%%%%%%%%%%%%%
%
%%%%%%%%%%%%%%%%%%%%%%%%%%%%%%%%%%%%%%%%%%

\subsection{Two species case}
\label{sec:5:2}
In this second test case we consider the multi-species framework, which is more relevant in
plasma physics but to reduce the computational
effort we design a simplified version of  the two species case. 
{Indeed, due to the fact that the present scheme does not cope with
  the time-stiffness, we
  have to adapt the time-step $\Delta t$ with the electron dynamics,
  thus choosing $\Delta t \sim \eps$. This leads (in general) to
  rather huge computational costs for the resolution of the kinetic
  ion dynamics for small $\eps$-values, the electron dynamics being
  not so cumbersome, due to the well-designed Hermite approach. Hence,
  to be able to perform some simulations in reasonable times, we shall
  suppose here the ions well-defined by macroscopic quantities, and
  uniquely the electrons follow a kinetic model. This shall  greatly
  accelerate the computations, and shall permit to focus on the
  adiabatic electron asymptotics, and the designed Hermite spectral
  approach.}
  %A fully performant AP-scheme shall be introduced in a next work, permitting more realistic simulations.}

Let us start by assuming the electron distribution function $f_e$
being solution to the Vlasov-Poisson-Fokker-Planck equation \eqref{Ele_syst}
with $(\nu_{ee},\nu_{ei})=(0.5,\,0.1)$  and an initial distribution given by 
\begin{align}
\label{bot}
f_e(0,x,v)\,=\,\frac{1}{\sqrt{2\,\pi}} \,\exp\left(-\frac{v^2}{2}\right)\,(1+\kappa\cos(k\,x))\,, \qquad \forall (x,v) \in {\mathbb T} \times \RR\,,
\end{align}
with $\kappa=0.01$ and $k=2\pi/L$ with $L=12$.

{The ion distribution function is supposed Maxwellian
  $$
  f_i(t,x,v):={n_i(x) \over \sqrt{2\, \pi\, T_i(t)}} \, e^{- {v^2
      \over 2\, T_i(t)}},
  $$
with a density depending only on the space variable and  given by}
$$
n_i(x) \,=\, 1\,+\, 0.2\,\cos(k\,x)\,, \qquad \forall x \in {\mathbb T}\,,
$$
whereas the mean velocity $u_i$ is set to zero and  the ion temperature $T_i$ is supposed to depend  only on the  time variable and to satisfy the following equation
$$
-\eps\, \mathfrak{n}_e\, \frac{\dD T_i}{\dD t} \,=\, \iint_{\TT\times\RR} \cQ_{ei}(f_e)\,|v|^2\dD v\dD x\,.
$$
The choice of the right hand side in the latter equation is motivated
by the requirement of an exact  conservation of the total energy. Indeed, multiplying
\eqref{Ele_syst} by $v^2/2$ and integrating with respect to
$(x,v)\in\TT\times\RR$, yields
\begin{eqnarray*}
\frac{\eps}{2}\,\frac{\dD }{\dD t} \iint_{\TT\times\RR} f_e \,|v|^2\,
\dD v\dD x &=& -\int_{\TT} E\,n_e\,u_e \dD x  \,+\,
               \frac{1}{2}\iint_{\TT\times\RR} \cQ_{ei}(f_e) \,|v|^2\,\dD
               v\dD x \\
   &=& -\int_{\TT} \phi\, \partial_x (n_e\,u_e) \dD x  \,+\,
               \int_{\TT}  \nu_{ei}\, n_e \, S_{ei}\,\dD x\,,
\end{eqnarray*}
where $S_{ei}$ is given by \eqref{def:Sei}.  Furthermore using the
continuity equation on $n_e$ \eqref{fluid} and the fact that $n_i$ does not depend on time,
we have
$$
\eps\,\partial_t(n_e - n_i) \,+\, \partial_{x} (n_e \,u_e) \,=0\,,
$$
such that using Poisson's equation and the equation for  $T_i$, 
permits indeed to obtain the total energy conservation
$$
\frac{\eps}{2}\,\frac{\dD }{\dD t} \left[\iint_{\TT\times\RR} f_e \,|v|^2\,
\dD v\dD x  \,+\, \int_{\TT} |E|^2 \dD x \,+\, \mathfrak{n}_e\,   T_i \right] \,=\, 0\,.
$$

\begin{remark}
Observe that similarly to $T_i$, we could also impose an equation
on the mean velocity $u_i$ to conserve the global momentum, but
it is not necessary for our purpose here. 
\end{remark}

Now let us verify that for this simplified model,  the distribution
function $f_e$ tends to the Maxwell-Boltzmann distribution \eqref{def:MB} when
$\eps\rightarrow 0$. The point is that the simplified model does not satisfy the
H-theorem for any $\eps>0$, but we will show that it is verified
when $\eps$ is sufficiently small provided that the macroscopic
quantities are bounded. Indeed, computing the time derivative of the
entropy gives
{\begin{eqnarray*}
\eps\,\frac{\dD }{\dD t} \iint_{\TT\times\RR} f_e \,\log(f_e)\,
\dD v\dD x &+& \iint_{\TT\times\mathbb R} \left[ \nu_{ee}\,
               T_{ee}\, {\cM_{e} \over h_e}
               \left|\partial_vh_e\right|^2
               \,+\,\nu_{ei}\,T_{ei}\frac{\cM_{ei}}{h_{ei}}
               \left|\partial_vh_{ei}\right|^2 \right]\dD v\,\dD
               x
  \\[1.1em]
  &=&  \iint_{\TT\times\mathbb R} \frac{\nu_{ei}}{T_{ei}}
      \left( f_e \,|v-u_{ei}|^2 - T_{ei}
      f_e\right)\dD v\dD x\,,
\end{eqnarray*}}
where $h_e$ and $h_{ei}$ are given in \eqref{def:h}. Then, computing the 
term on the right hand side yields
{\begin{eqnarray}
   \nonumber
\eps\, \frac{\dD\cH_e}{\dD t} &+& \iint_{\TT\times\mathbb R} \left[ \nu_{ee}\,
               T_{ee}\, {\cM_{e} \over h_e}
               \left|\partial_vh_e\right|^2
               \,+\,\nu_{ei}\,T_{ei}\frac{\cM_{ei}}{h_{ei}}
               \left|\partial_vh_{ei}\right|^2 \right]\dD v\,\dD
                                       x \\
   \label{eq:123}
    &=& -\frac{1-\eps^2}{1+\eps^2}\,\int_{\TT} \frac{\nu_{ei}\,n_e |u_e|^2}{4
               \,T_{ei}} \dD x \,+\, \frac{\eps^2}{1+\eps^2} \int_{\TT} \frac{\nu_{ei}\,n_e}{T_{ei}} \, \left(T_e^{\eps}-T_i\right)\, \dD x.  
\end{eqnarray}}
Therefore, in the limit $\eps\rightarrow 0$, we can proceed as in the
proof of Theorem \ref{th:1} and get, from the entropy dissipation,  that  in the limit $f_e$ tends towards an equilibrium of the form $\cM_{n_e,0,T_e}$ where $n_e$ is given by \eqref{def:MB}.
Furthermore, using the definition of $S_{ei}$ in \eqref{def:Sei},
the equation on $T_i$ can be written as
$$
\mathfrak{n}_e\, \frac{\dD T_i}{\dD t} \,=\, \frac{2\,\eps}{1+\eps^2}\,\int_{\TT}
\nu_{ei}\,n_e\left[ T_e- T_i + \frac{|u_e|^2}{2}\right]\dD x\,,
$$
which means that $T_i$ converges to a constant temperature as
$\eps\rightarrow 0$. {Let us mention that a similar approach has been
used in \cite{Buet}  in a slightly different context.} \\

After this short presentation of our simplified model, let us present our numerical results. We take $N_x\times N_H=32\times 32$ and compare the obtained
solutions  with a reference solution computed on a refined grid of $N_x\times
N_H=128\times 128$ for several values of $\eps\in\{
10^{-3},\,10^{-2},\,10^{-1},\,1 \}$. The scaling parameter $v_{th}$ in
\eqref{c-n} is chosen
so that the numerical scheme captures the asymptotic
behavior of the solution $f_e$ when $\eps\rightarrow 0$. Since the
initial data is far from the Maxwell-Boltzmann equilibrium, the time
step is initially chosen proportional to $\eps$ as $\Delta t =
\eps/500$ whereas we choose $N_x=N_H=32$.

We show first  the numerical results for $\eps=10^{-3}$.  On one hand, we present on Figure \ref{fig:21} $(a)$ the time evolution of the deviations of the discrete mass and total energy, when compared with the initial values. Here, the
errors on mass and total energy are of order $10^{-7}$.  We remind
that our space discretization does not ensure exact conservation of energy but their variations remain very small during the
simulation. On the other hand, we  present the time evolution of the
$L^2$-norm of the first Hermite coefficients $(\alpha_k)_{1\leq k
\leq 6}$ in Figure \ref{fig:21} (b).  For $k\geq 1$, the $L^2$-norm of these coefficients decreases almost exponentially
fast and oscillates. These numerical results illustrate the
efficiency of our approach based on the Hermite decomposition of the
distribution function $f_e$, when treating situations with $\eps
\ll 1$. Indeed, for $\eps\rightarrow 0$ all
  coefficients $(\alpha_k)_{k\geq 1}$ converge to zero very
  rapidly as $e^{-Ct/\eps}$, hence after a short transient regime, the numerical solution can be approximated very well with only few Hermite coefficients, the Maxwell-Boltzmann
distribution corresponding to only one coefficient
$\alpha_0$. Once again, our Hermite decomposition is
particularly well adapted to this asymptotic regime since the number of modes may be adapted
along the simulation by neglecting the Hermite coefficients of smaller
amplitudes \cite{vencels}.  The main issue remains however to develop an efficient
time discretization avoiding the $\eps$-dependent constraint on the time step. \\

\begin{figure}
  \centering
  \begin{tabular}{cc}
	\includegraphics[width=3.2in,clip]{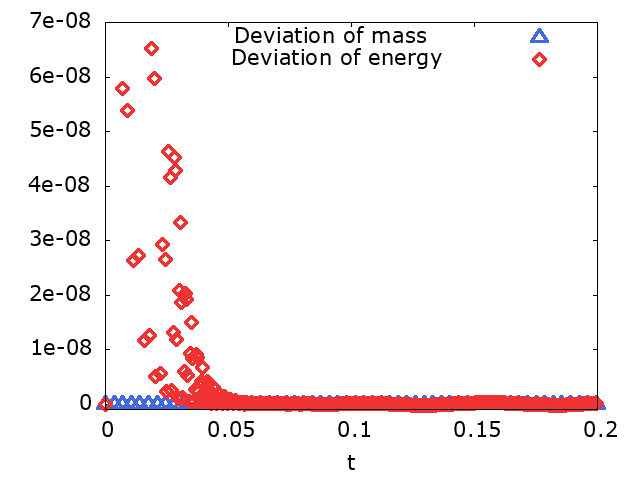} &
	\includegraphics[width=3.2in,clip]{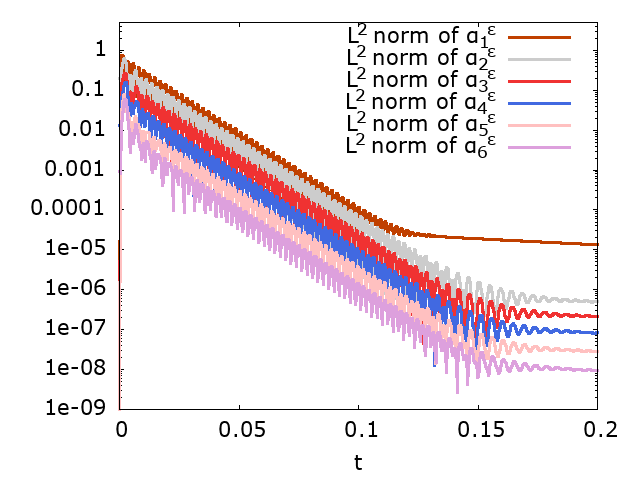}
        \\
        (a) & (b)
 \end{tabular}
\caption{{\bf Two species case ($\eps=10^{-3}$):} $(a)$ deviation of mass and energy with respect to the initial condition; $(b)$ time evolution of the  $L^2$ norm of the Hermite coefficients  $(\alpha_k)_{1\leq k \leq 6}$ in logarithmic scale for  $N_x\times N_H = 32 \times 32$.}
\label{fig:21}
\end{figure}

We also plot the time evolution of the potential
energy  and of the global temperatures (averaged in space) of the electrons and ions in Figure
\ref{fig:22}.  We compare them to the results
obtained  with a refined mesh of size $128\times 128$ and we can see that
these results have the same structure. This means that with coarse grids 
we already get satisfactory results. Furthermore,
  we observe that the potential energy oscillates in time and is damped until it
  reaches a stationary state. The electron
  temperature has a similar behavior whereas the ion temperature grows
  until it converges finally also towards a stationary state.

\begin{figure}
  \centering
  \begin{tabular}{cc}
    	\includegraphics[width=3.2in,clip]{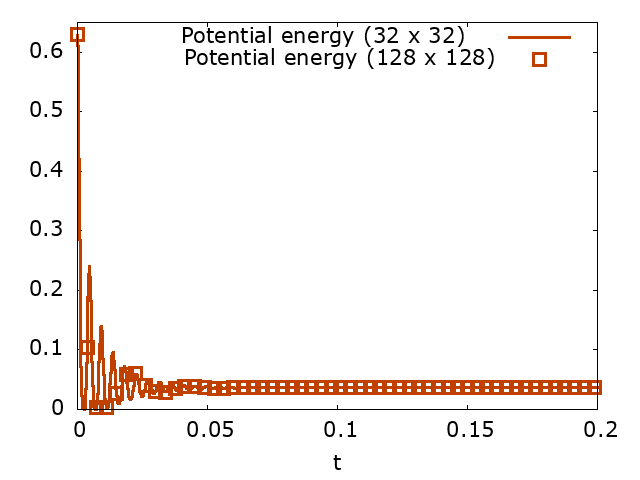}&
	\includegraphics[width=3.2in,clip]{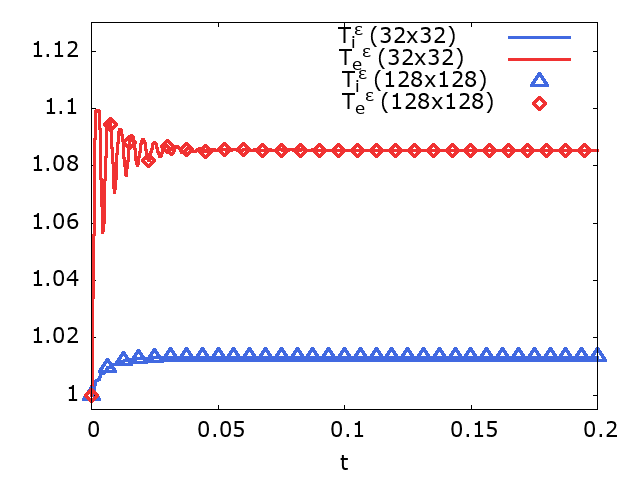}
        \\
    (a) & (b)
 \end{tabular}
	\caption{{\bf Two species case ($\eps=10^{-3}$):} $(a)$ time
          evolution of the potential energy, $(b)$ time evolution of
          the global temperatures $T_e$ and $T_i$ with $N_x\times N_H = 32 \times 32$, whereas the reference solution is with $N_x\times N_H=128\times 128$.}
	\label{fig:22}
\end{figure}

Finally in order to illustrate the convergence to the Maxwell-Boltzmann
distribution, we show  some snapshots of the electron density $n_e$
and of the electric potential $\phi$  in Figure \ref{fig:23}.  Both
quantities converge, after an oscillatory transient region, towards their
equilibrium corresponding to \eqref{def:MB}-\eqref{cstr}.\\

\begin{figure}
 \centering
 \begin{tabular}{cc}
 \includegraphics[width=3.2in,clip]{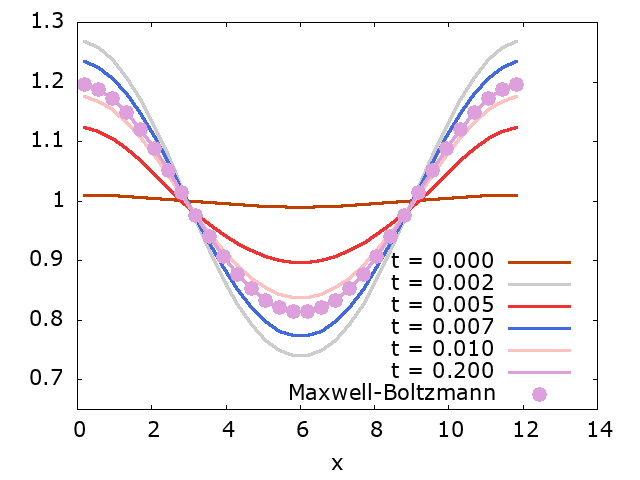} &
 \includegraphics[width=3.2in,clip]{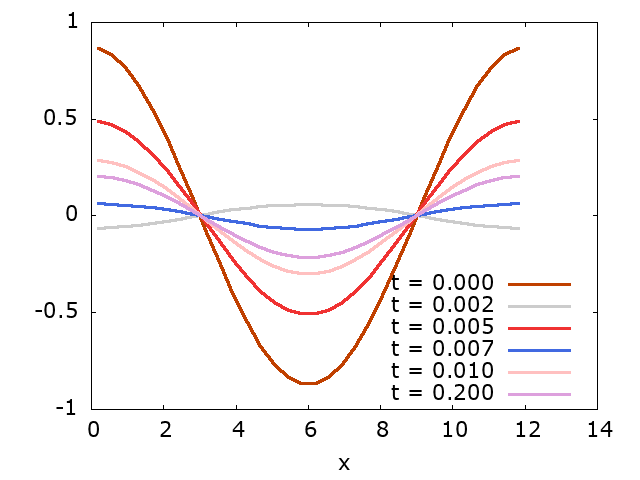}
 \\
 (a) & (b)
\end{tabular}
\caption{{\bf Two species case ($\eps=10^{-3}$):}  plot of the (a) electron density $n_e$ and (b) electric potential $\phi$ at different times and with  $N_x\times N_H = 32 \times 32$.}
\label{fig:23}
\end{figure}

Next we performed some computations for other values of
$\eps>0$. We get similar results concerning the mass and
energy variations. To illustrate the different regimes, we present in Figure \ref{fig:24}, the time evolution
of the electron temperature $T_e$ resp. ion temperature $T_i$ as well as of the Hermite coefficients
$(\alpha_k)_{1\leq k \leq 6}$ for $\eps=10^{-1}$ and
$\eps=1$.  For such large values of $\eps$ we are no more close to the Maxwell-Boltzmann regime. When $t$ becomes larger, the solution $f_e$ of
\eqref{Ele_syst} converges to a steady state, and both temperatures,
after oscillating, tend towards a stationary state.  When
$\eps=1$ (and $\nu_{ee}=0.1$), the electronic temperature strongly oscillates, then it
approaches an equilibrium at $t \simeq 40$. However, for $\eps=0.1$,
oscillations are rapidly damped and the cooling process is much slower
than in the previous case, in particular $T_e$ approaches its equilibrium at
$t\simeq 120$. {This point comes from the fact that the temperature equilibration is   $\varepsilon$-dependent, and for smaller and smaller $\epsilon$
  values, the equlibration of the ion and electron temperatures get
  slower and slower, however, for each $\varepsilon >0$ both temperatures converge towards the same
  value in the
  long-time limit.}
\\

Concerning the Hermite coefficients, they first oscillate with a damping amplitude,
but after a while they stabilize ($\eps=1$) or increase slowly
($\eps=10^{-1}$) . This underlines the fact that even if the solution $f$
converges to an equilibrium when $t$ goes to infinity, this equilibrium does not
coincide with the Maxwell-Boltzmann distribution obtained when
$\eps\rightarrow 0$. The two limits are different in the here presented test case.
Therefore, the solution cannot be represented by
only one coefficient ($\alpha_0$) but higher-order even coefficients are
mandatory during the simulations whereas odd coefficients converge to zero.

\cla{However, we observe that most coefficients decrease rapidly to zero
and the efficiency of the adpative algorithm proposed in Section
\ref{sec:4.3} is illustrated in Figure \ref{fig:25}, where we present
the time evolution of the number of considered Hermite functions for various
values of $\eps$, ranging  from $10^{-3}$ to $1$. As expected, when $\eps$ is
small ($\eps=10^{-3}$ or $10^{-2}$), the solution converges rapidly to
the Maxwell-Boltmzann distribution after a short initial time layer
and $f$ can be approximated with very few Hermite
coefficients ($0\leq k\leq 2$ since we apply the adaptive algorithm
only for $k\geq 3$). For larger values of $\eps$, the solution does not match
with the Maxwell-Boltzmann distribution and the number of Hermite coefficients
oscillates according to the variations of the distribution function,
but after some time only a small number of coefficients are finally used ($0\leq k\leq 5$). This
algorithm allows to reduce drastically the computational time for
various regimes.}

%To precise this point, let us only mention that for fixed $\eps \in [0,1]$, letting $t \rightarrow \infty$ in \eqref{eq:kin}-\eqref{constr-phi} leads to the same stationary distribution functions for electrons as well as for  ions, given by
%$$
%f_\infty (x,v):= {{\mathfrak N} \over \sqrt{2 \pi \, T_\infty}} \, e^{- {v^2 \over 2\, T_\infty}}\,, \qquad T_\infty:= {{\mathfrak E} \over {\mathfrak N}}\,.
%$$
%Trying now to expand this stationary distribution function $f_\infty$ with the help of the basis functions $\{\psi_k(t,\cdot)\}_{k \in \NN}$ introduced in \eqref{Heri} and used in our DG-H scheme, namely
%$$
%f_\infty(x,v):= \sum_{k=0}^{\infty} \alpha_k(t,x)\, \psi_k(t,v)\,,
%$$
%requires to take into account for several coefficients $\alpha_k(t,x)$. Indeed, the basis functions $\{\psi_k(t,\cdot)\}_{k \in \NN}$ are adapted to the adiabatic limit $\eps \rightarrow 0$, forming a complete, orthonormal basis of $L^2(\cM_{v_{th}}^{-1}\, \dD v)$, with the chosen weight given by the ``limiting'' electron Maxwellian
%$$
%\cM_{v_{th}(t)}(v)\,:=\,{ 1 \over \sqrt{2 \pi}\, v_{th}(t)}  \exp\left(-{v^2 \over 2 v_{th}^2(t)}\right)\,, \qquad v_{th}(t) := \sqrt{T(t)}\,,
%$$
%with a time-dependent electron temperature $T(t)$ given by the asymptotic model \eqref{def:Te}-\eqref{cstr}.

\begin{figure}
  \centering
  \begin{tabular}{cc}
       \includegraphics[width=3.2in,clip]{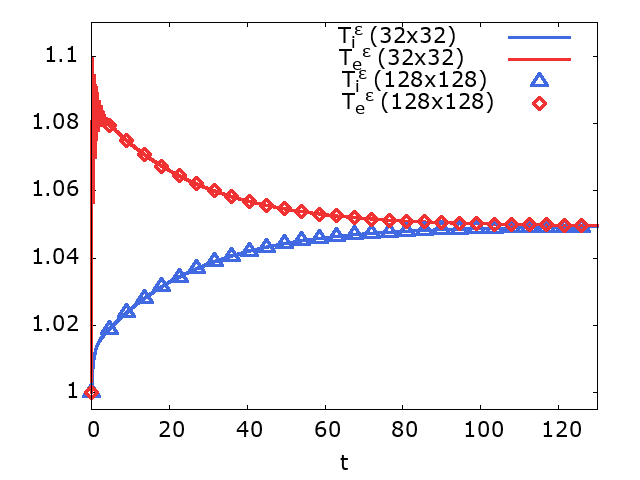}&
       \includegraphics[width=3.2in,clip]{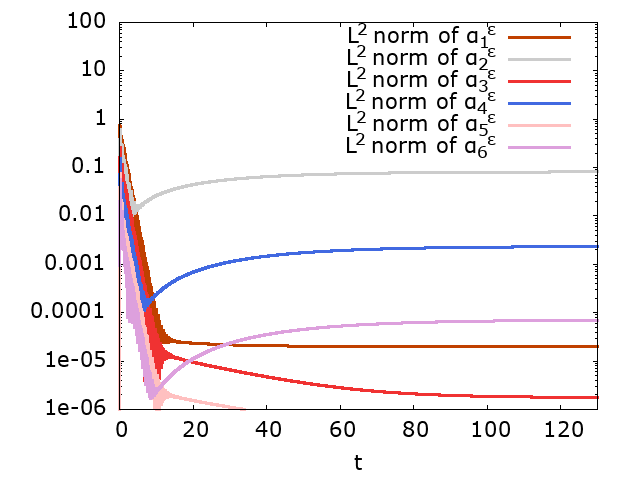}\\
 (a) Temperature ($\eps=0.1$)& (b) $(\alpha_k)_{1\leq k\leq 6}$
                               ($\eps=0.1$) \\
    \includegraphics[width=3.2in,clip]{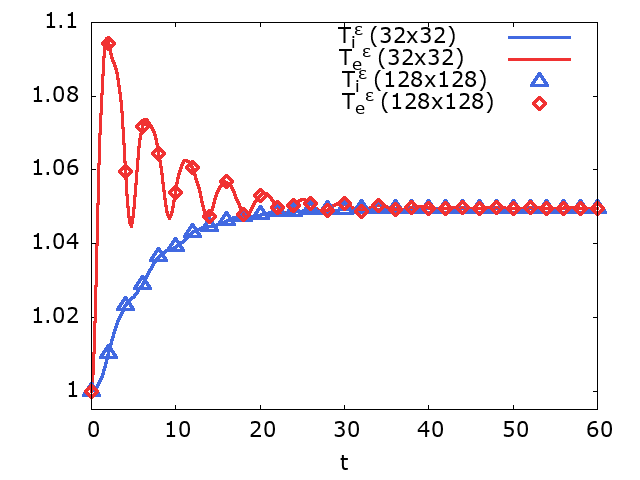}&
       \includegraphics[width=3.2in,clip]{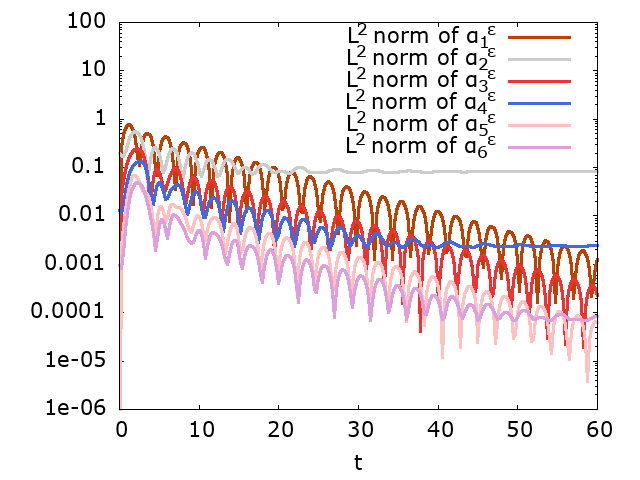} \\
    (c) Temperature ($\eps=1$)& (d) $(\alpha_k)_{1\leq k\leq 6}$  ($\eps=1$)
 \end{tabular}
	\caption{{\bf Two species case $\eps=0.1$ and $\eps=1$:}
         time evolution of  (left) 
          the global temperatures $T_e$ and $T_i$  (right) the  $L^2$ norm of the Hermite coefficients  $(\alpha_k)_{1\leq k \leq 6}$ in logarithmic value.}
	\label{fig:24}
\end{figure}

\begin{figure}
  \centering
  \begin{tabular}{cc}
    \includegraphics[width=3.2in,clip]{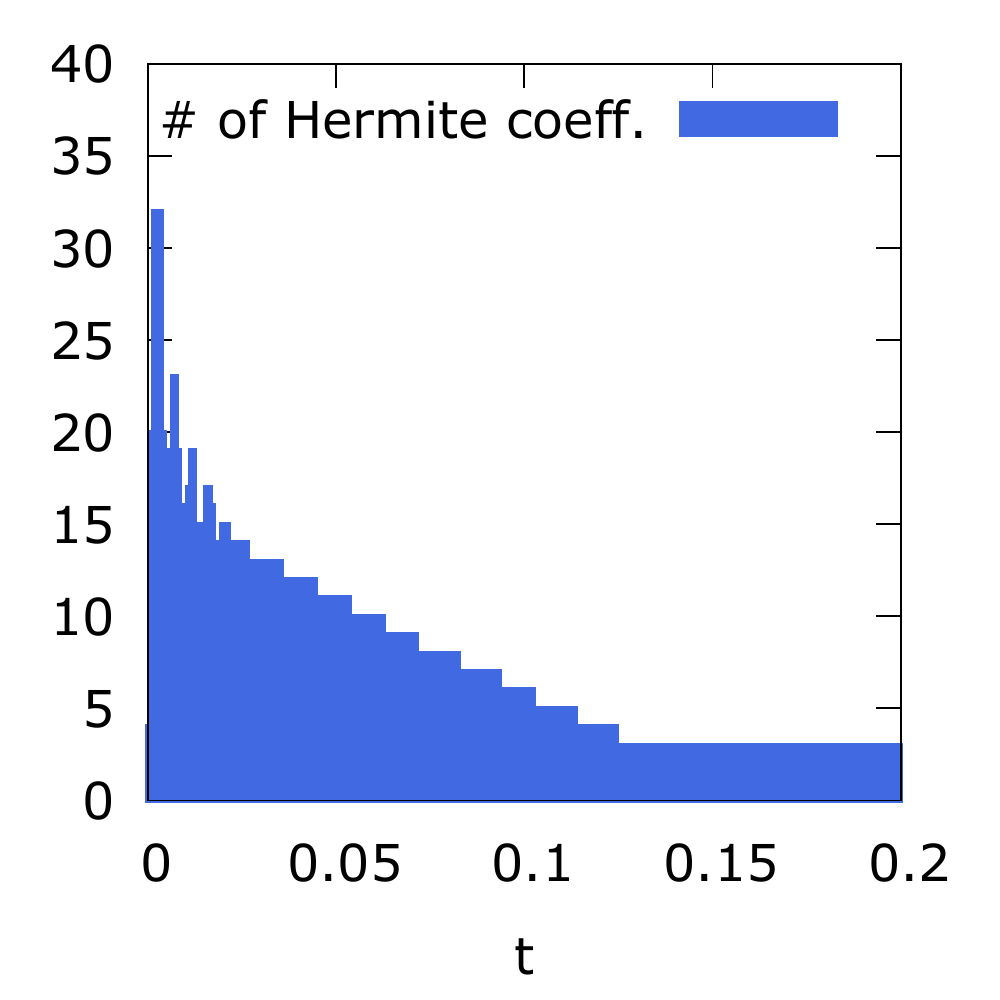}&
    \includegraphics[width=3.2in,clip]{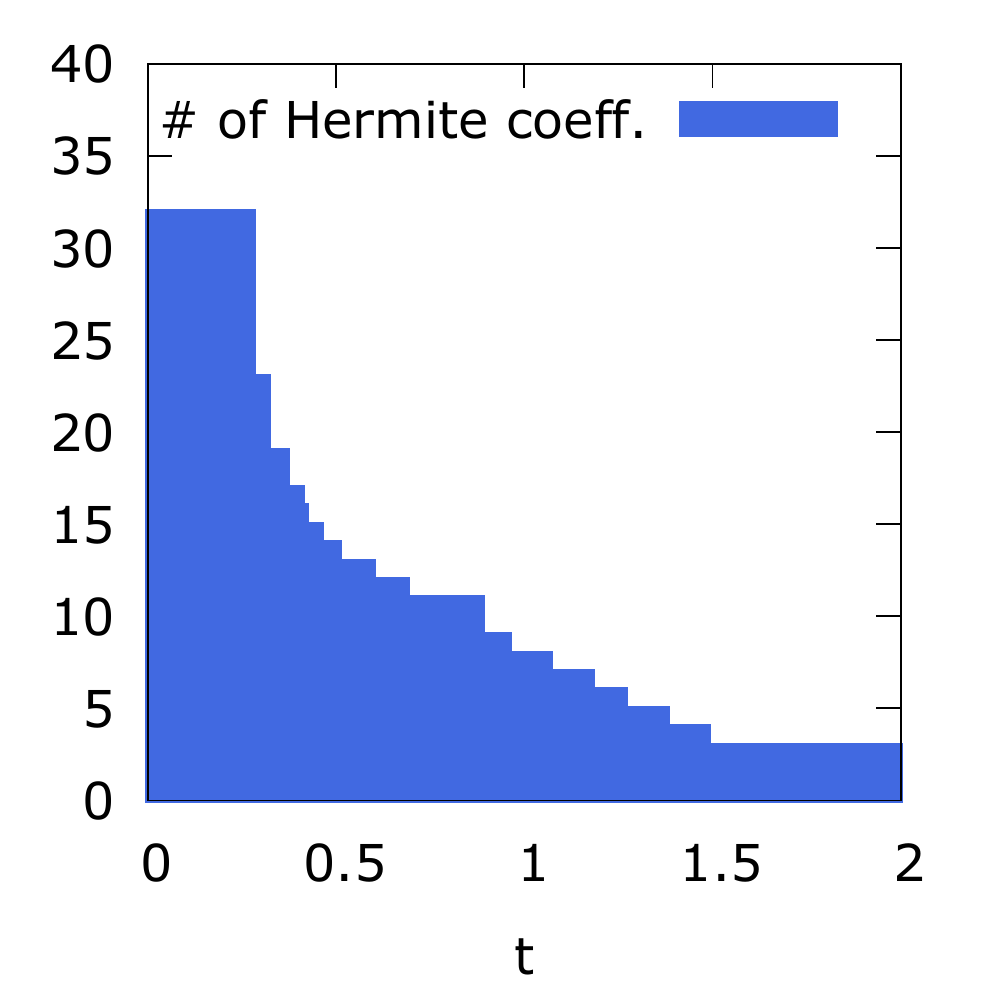}\\
    (a)  $\eps=10^{-3}$ & (b)  $\eps=10^{-2}$\\
    \includegraphics[width=3.2in,clip]{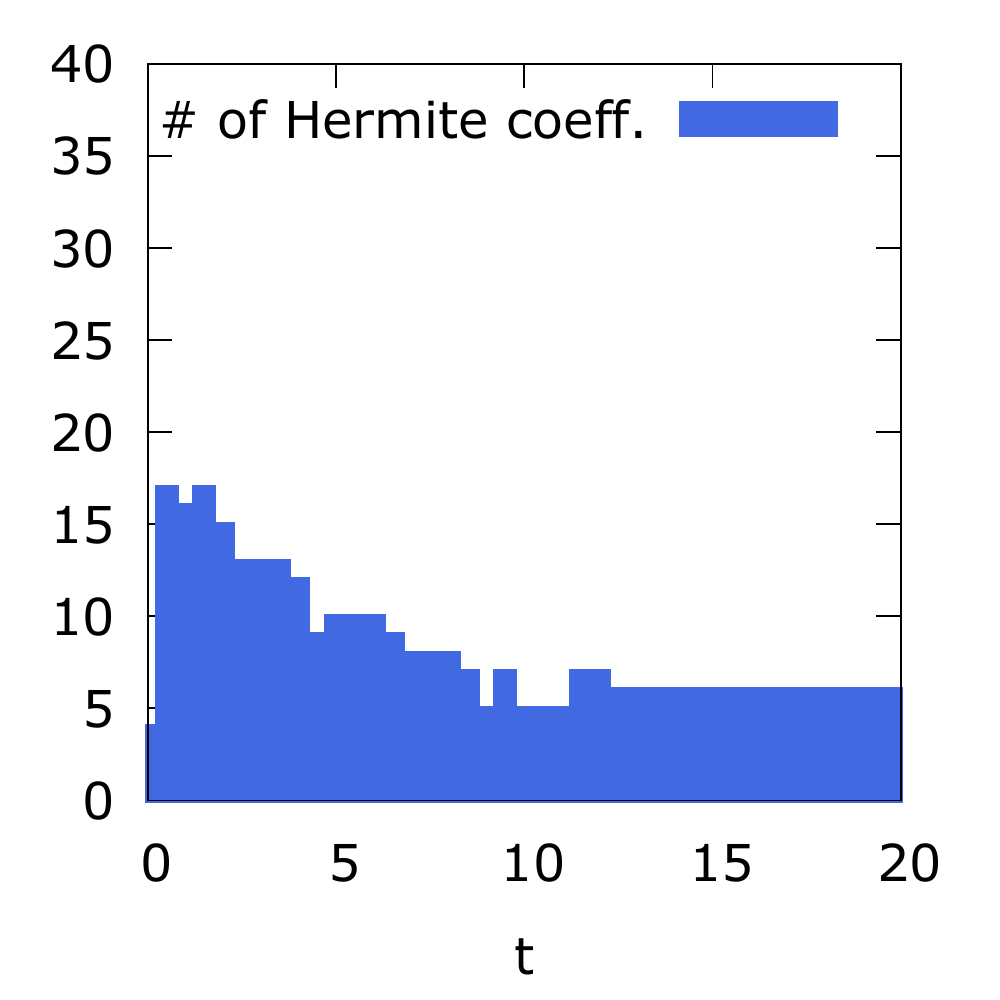}&
    \includegraphics[width=3.2in,clip]{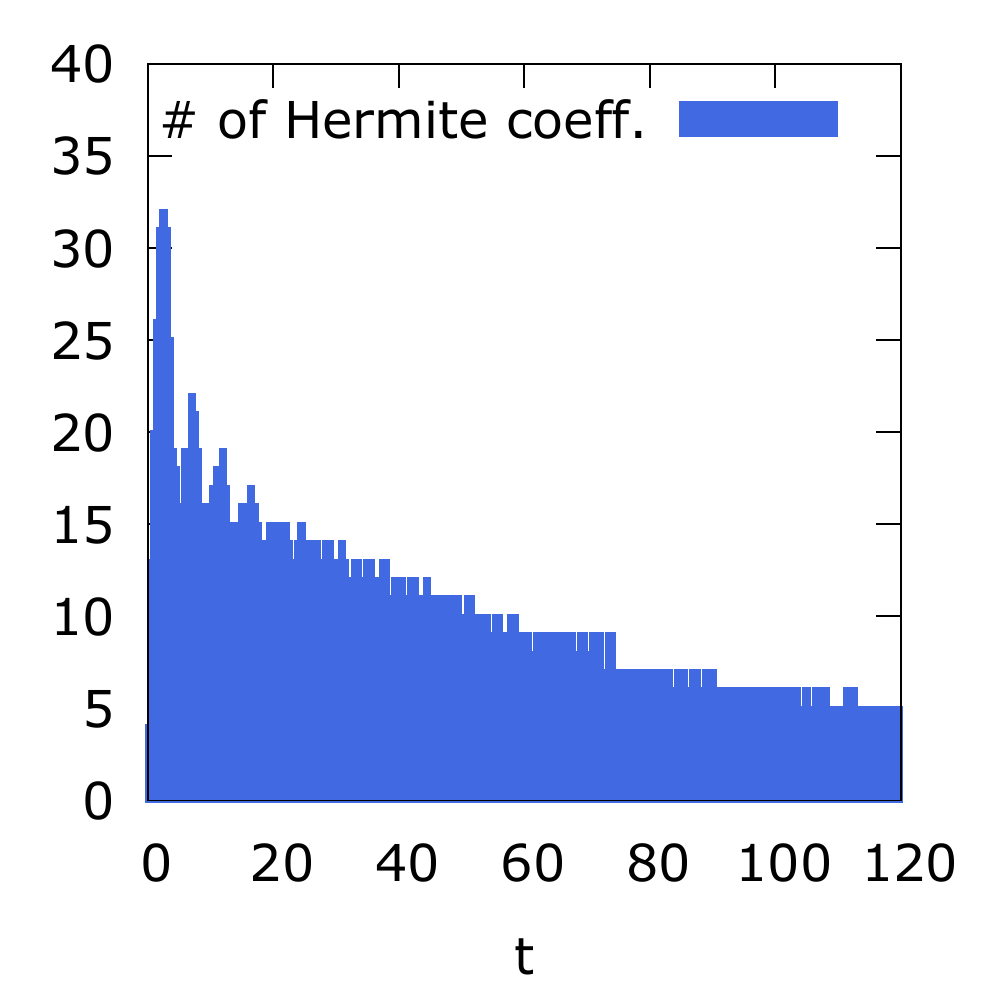}\\
 (c)  $\eps=10^{-1}$ & (d) $\eps=1$.
 \end{tabular}
	\caption{{\bf Two species case :}
         time evolution of the number of considered Hermite coefficients when
         Algorithm \ref{alg:1} is applied for different values of $\eps$.}
	\label{fig:25}
\end{figure}

%%%%%%%%%%%%%%%%%%%%%%%%%%%%%%%%%%%%%%%%%%%%%%%%%%%%%%%%%%%%%%%%%%%
\section{Concluding remarks and perspectives}
\label{SECC}
\setcounter{equation}{0}
\setcounter{figure}{0}
\setcounter{table}{0}

%%%%%%%%%%%%%%%%%%%%%%%%%%%%%%%%%%%%%%%%%%%%%%% 

Let us conclude this paper by summarizing what was achieved in this work and what remains still to be done in future works.\\

The focus of this paper was the introduction of mixed Fokker-Planck collision operators taking into account especially for ion-electron collisions in thermonuclear fusion plasmas, and satisfying the desired physical properties as the three conservation laws and the entropy-decay relation. Based on these new operators, a second aim was to study the adiabatic electron limit $\eps \rightarrow 0$, where  the small parameter $\eps$ stands somehow for the electron-to-ion mass ratio. The small $\eps$-regime corresponds to the description of phenomena occurring at ion scales, whereas the rapid electrons are thermalized and approximated via macroscopic models (adiabatic Boltzmann relation). 

A formal asymptotic limit $\eps \rightarrow 0 $ permitted to obtain
the macroscopic model satisfied by the thermalized electrons. During
this limit the ions remain kinetic. Then a first numerical
scheme was proposed in order to solve the Vlasov-Poisson-Fokker-Planck
system, based on a Hermite spectral method in the velocity variable
and a discontinuous Galerkin method in the space variable. \\ 

One of the main difficulties when trying to solve numerically the Vlasov-Poisson Fokker-Planck system \eqref{eq:kin}-\eqref{constr-phi} in the small $\eps$-regime is the singularity of the problem. The advantage of choosing a Hermite spectral method (to discretize the velocity variable) with a suitable choice of the weights, is that in the small $\eps$-regime only few modes have to be taken into account. This reduces drastically the computational costs. Indeed, an exact Maxwellian, as our limiting adiabatic distribution function, is fully represented by only one mode in the Hermite expansion, if the scaling is well adapted. Thus the transition from the kinetic to the adiabatic model is somehow intrinsic to this Hermite spectral approach.\\

There remain however still several points  to be  treated in future
works, to render the method more efficient.  For example for weakly
collisional plasmas, the dissipation is not sufficiently large and the
time step still depends on $\eps$, only the Hermite-approach permitted
to render the computations more tractable for the electrons. \cla{A {\it
  multi-scale} approach is one of our next aims on the way
to get more performant methods in the study of this two-species adiabatic
limit, in particular to be able to adapt dynamically the time step
(initial layer, equilibrium regime) or to choose $\eps$-independent grids.} Furthermore, a tricky discretization of the Limit-model,
permitting to compute efficiently  the scaling factor $v_{th}(t)$, is
also a complex task to be achieved. And finally, after all these
improvements, a full ion/electron computation shall become possible
and has to be completed in a real physical situation \cite{CF_bis}.

\bigskip

%%%%%%%%%%%%%%%%%%%%%%%%%%%%%%%%%%%%%%%%%%%%%%%%%%%%%%%%%%%%%%%%%%%
\noindent {\bf Acknowledgments.} % { The authors are grateful to the COVID period, forcing them to remain blocked in Toulouse and  enabling thus this collaboration, as well as to the IMT, which allowed them to work during this sad period, and supported calmly  the authors screams.}
This work has been carried out within the framework of the EUROfusion
Consortium, funded by the European Union via the Euratom Research and
Training Programme (Grant Agreement No 101052200 — EUROfusion). Views
and opinions expressed are however those of the author(s) only and do
not necessarily reflect those of the European Union or the European
Commission. Neither the European Union nor the European Commission can
be held responsible for them.
%%
%% 
%%
%% 
%%  thanks the support  project ANR-19-CE46-0004.
%%%%%%%%%%%%%%%%%%%%%%%%%%%%%

%%%%%%%%%%%%%%
\end{document}